\documentclass[11pt]{article}

\usepackage[T1]{fontenc}
\usepackage[utf8]{inputenc}

\usepackage[a4paper,margin=1in]{geometry}
\usepackage{microtype}

\usepackage{amsmath,amssymb,amsfonts,amsthm,mathtools,bm,mathrsfs}
\usepackage{enumitem}
\usepackage{booktabs,multirow,array}

\usepackage{graphicx}
\usepackage{subcaption}
\usepackage{caption}
\usepackage{float}

\usepackage[colorlinks=true,linkcolor=blue,citecolor=blue,urlcolor=blue]{hyperref}

\numberwithin{equation}{section}
\theoremstyle{plain}
\newtheorem{theorem}{Theorem}[section]
\newtheorem{lemma}[theorem]{Lemma}

\theoremstyle{definition}
\newtheorem{definition}[theorem]{Definition}
\newtheorem{example}[theorem]{Example}
\theoremstyle{remark}

\newenvironment{keywords}{\par\vspace{0.5em}\noindent\textbf{Keywords.}\ }{\par\vspace{0.5em}}

\title{High-Order Regularity and a Fully Discrete Fourier Spectral Method for a Partially Dissipative Viscoelastic Timoshenko System with Memory}
\author{Zhenyang Zhong\thanks{School of Science, Harbin Institute of Technology, Shenzhen, Guangdong Province 518055, China. Email: \texttt{zhongzhenyang8.20@foxmail.com}.}
\and Hui Liang\thanks{School of Science, Harbin Institute of Technology, Shenzhen, Guangdong Province 518055, China. Email: \texttt{lianghui@hit.edu.cn}. Corresponding author.}}
\date{}

\begin{document}

\maketitle

\begin{abstract}
This paper investigates a class of partially dissipative viscoelastic Timoshenko systems with memory, where dissipation is induced by a Volterra-type memory term acting only on the shear variable. The well-posedness of weak and strong solutions is established on finite time intervals, including existence, uniqueness, stability, and higher-order regularity under compatibility conditions consistent with mixed boundary conditions. For the numerical approximation, a Fourier spectral fully discrete scheme is constructed: sine and cosine basis expansions are used in space for unknowns satisfying Dirichlet and Neumann boundary conditions, respectively; in time, a central difference scheme is applied to the second-order derivatives, and the composite trapezoidal rule is used to approximate the memory convolution term. Based on a discrete energy method, the positivity of the constructed discrete energy is proved, and the error estimate for the fully discrete scheme with second-order convergence in time and \(q\)-th order in space is established for any $q \in \mathbb{N}$.
Numerical experiments are given to verify the theoretical convergence rates and to compare the dynamic responses of the local and nonlocal models, demonstrating that the memory term effectively captures energy dissipation and vibration attenuation behavior in viscoelastic materials.
\end{abstract}
\section{Introduction}
Timoshenko beam models are widely used to describe beam vibrations and wave propagation when shear deformation and rotational inertia cannot be neglected \cite{timoshenko1921correction,han1999dynamics}. In viscoelastic materials, hereditary effects are commonly modeled through memory kernels, which leads naturally to Timoshenko systems with memory \cite{christensen2013theory}.

In this paper, we consider the partially dissipative viscoelastic Timoshenko system
\begin{align}\label{Timoshenko}
   \left\{\begin{array}{ll}
   \rho_1\phi_{tt}-\kappa\eta_x+\kappa\displaystyle\int_0^tg(t-s)\eta_x(s)\,\mathrm{d}s=0,
   & (x,t)\in(0,l)\times(0,T],\\[1.2ex]
   \rho_2\psi_{tt}-b\psi_{xx}+\kappa\eta-\kappa\displaystyle\int_0^tg(t-s)\eta(s)\,\mathrm{d}s=0,
   & (x,t)\in(0,l)\times(0,T],
   \end{array}\right.
\end{align}
where $\eta:=\phi_x+\psi$, supplemented with the mixed boundary conditions
\begin{align}\label{BC}
    \phi(0,t)=\phi(l,t)=0,\qquad \psi_x(0,t)=\psi_x(l,t)=0,\qquad t\in [0,T],
\end{align}
and the initial data
\begin{align}\label{IC} 
\left\{
\begin{array}{l}
\phi(x,0)=\phi_0(x),\qquad \phi_t(x,0)=\phi_1(x),\\
\psi(x,0)=\psi_0(x),\qquad \psi_t(x,0)=\psi_1(x),\qquad x\in(0,l).
\end{array}
\right.
\end{align}
Here, \(\phi  = \phi (x, t)\) and \(\psi  = \psi (x, t)\) are respectively the vertical displacement and rotation angle, 
while  the structural coefficients \(\rho_1,\rho_2,\kappa,b,T,l\)
are positive constants that depend on the structure of the beam material.
We assume that the memory kernel \(g:\mathbb{R}^+\to\mathbb{R}^+\) is nonincreasing, continuously differentiable, and satisfies
\begin{align}\label{g}
    g(0)>0,\qquad 1-\int_0^\infty g(s)\,\mathrm{d}s>0.
\end{align}
The system \eqref{Timoshenko} was introduced in \cite{alves2019modeling}, where uniform energy decay was established and the equal-wave-speed condition $\frac{\kappa}{\rho_1}=\frac{b}{\rho_2}$ was shown to be necessary and sufficient for uniform decay. More generally, Timoshenko systems with damping or memory have been extensively studied from the viewpoints of stability and asymptotic behavior (see \cite{ammar2003energy,alves2019modeling,tavaresSilvaMaOquendo2024shearing}). Related viscoelastic and fractional Timoshenko models have also been investigated with respect to well-posedness, regularity, and numerical approximation (see \cite{zheng2022viscoelastic,li2023viscoelastic}). 
Nevertheless, the present system still lacks a complete well-posedness and regularity framework for rigorous numerical analysis.
In particular, the existence and uniqueness results in \cite[Theorem 3.2]{alves2019modeling} are only outlined via a Faedo--Galerkin method, without detailed formulations of weak and strong solutions, compatibility conditions, or higher-order regularity. 
Moreover, although substantial progress has been made on the numerical analysis of evolution equations with memory, fractional wave equations, and viscoelastic wave models
\cite{mustaphaMclean2009memoryDG,xu2015decayMemory,guoJiangXiongZhang2024constantQ}, rigorous convergence analyses for high-accuracy fully discrete numerical methods tailored to the partially dissipative Timoshenko system \eqref{Timoshenko} are still limited.

A distinctive feature of \eqref{Timoshenko} is that dissipation is induced solely by the Volterra-type memory term acting on the shear variable $\eta=\phi_x+\psi$. Consequently the system is partially dissipative, with damping transmitted indirectly through the coupling structure. In addition, the memory term is nonlocal in time, and the mixed boundary conditions require different spatial treatments for the two unknowns. These features make both the continuous energy analysis and the design of accurate fully discrete methods more delicate than for standard local Timoshenko systems (see \cite[Eq.~(1.1)]{messaoudi2008internal}).

The aim of this paper is twofold. First, a complete well-posedness and regularity theory is established for \eqref{Timoshenko}--\eqref{IC} on finite time intervals. The precise definitions of weak and strong solutions are given. The existence, uniqueness, and stability of weak solutions are established first, and the corresponding well-posedness result for strong solutions is then derived under additional regularity and compatibility assumptions; the higher-order regularity estimate is further derived under suitable regularity assumptions on the initial data and compatibility conditions consistent with \eqref{BC}. Second, a Fourier spectral fully discrete method adapted to the structure of the problem is developed. In space, a sine expansion is used for $\phi$ 
and a cosine expansion for $\psi$ to incorporate naturally the mixed boundary conditions. 
In time, the centered difference for the second-order derivatives is combined with the composite trapezoidal rule for the memory integral. By constructing a suitable discrete energy for the error system, its positivity is proved under a suitable mesh condition and the error estimate is derived with second-order accuracy in time and \(q\)-th order accuracy in space for sufficiently smooth solutions with any $q \in \mathbb{N}$. Numerical experiments confirm the theoretical convergence rates and illustrate the dissipative effect of the memory.

The remainder of the paper is organized as follows. In Section~\ref{s2}, the weak and strong solutions of \eqref{Timoshenko}--\eqref{IC} are defined, the well-posedness of weak solutions is established, and the second-order energy estimate needed for the well-posedness of strong solutions is derived. In Section~\ref{s3}, this estimate is  generalized to arbitrary orders and used to prove the higher-order regularity. In Section~\ref{s4}, a fully discrete Fourier spectral scheme is constructed, a discrete energy is derived from the error equation, and the convergence of the scheme is proved. In Section~\ref{s5}, numerical experiments verify the convergence theory and illustrate the dissipative and memory effects of the nonlocal model. 
Directions for future work are presented in Section~\ref{s6}. 

Throughout the paper, $C$ denotes a generic positive constant, independent of the variables under consideration, which may take different values in different places.  The notation $A\lesssim B$ means that $A\leq CB$ for such a constant $C$.
In $L^2:=L^2(0,l)$, we write $(\cdot,\cdot)$ and $\|\cdot\|_2$ for the inner product and norm, respectively.

\section{Well-posedness of weak and strong solutions}\label{s2}
In this section, we introduce the function spaces used throughout the paper and establish the well-posedness of weak and strong solutions to \eqref{Timoshenko}--\eqref{IC} on finite time intervals.

First, we introduce several standard function spaces. Let $q$ be a positive integer, $H^q:=H^q(0,l)$, $\|u\|_{H^q}^2:=\sum_{k=0}^q \| u^{(k)} \|_2^2$, and
\begin{align*}
&L_*^2:=L_*^2(0,l):=\Big\{u\in L^2(0,l);\ \frac1l\int_0^l u(x)\,\mathrm{d}x=0\Big\},\\
&H_0^1:=H_0^1(0,l):=\left\{u\in H^1(0,l);\ u(0)=u(l)=0\right\},\\
&H_*^1:=H_*^1(0,l):=\Big\{u\in H^1(0,l);\ 
\frac1l\int_0^l u(x)\,\mathrm{d}x=0\Big\},\\
&C_c^\infty(0,T):=\left\{\varphi:(0,T)\to\mathbb{R};\ \varphi\in C^\infty(0,T),\ \operatorname{supp}\varphi \subset (0,T)\ \text{is compact} \right\}.
\end{align*}
By the Poincar\'e inequality, 
the following equivalent norms may be adopted on 
\(H_0^1\) and \(H_*^1\), respectively:
\[
\|u\|_{H_0^1}:=\|u_x\|_2,
\qquad
\|u\|_{H_*^1}:=\|u_x\|_2.
\]

Let $I_T:=[0,T]$ and let $(X,\|\cdot\|_X)$ be a Banach space. 
We denote by $C^q(I_T;X)$ the space of all $X$-valued functions $\mathbf{v}: I_T\to X$ 
that are $q$ times continuously differentiable, 
equipped with the norm (see \cite{RaviartThomas1983}, p.~156)
\[
\|\mathbf{v}\|_{C^q(I_T;X)} :=\max_{0\leq k\leq q}\left\{\sup_{t\in I_T}
\left\|\frac{\mathrm{d}^k\mathbf{v}}{\mathrm{d}t^k}(t)\right\|_X\right\}
<\infty.
\]

We now introduce the definition of weak and strong solutions. 
\begin{definition}\label{weak_sol}
We say that a pair of functions \((\phi,\psi)\) is a weak solution of the Timoshenko system \eqref{Timoshenko}--\eqref{IC} if it satisfies the regularity conditions
\begin{align*}
&\phi \in C(I_T;H_0^1) \cap C^1(I_T;L^2), \quad \psi \in C(I_T;H_*^1) \cap C^1(I_T;L_*^2),
\end{align*}
and, for every \( w \in H_0^1, v \in H_*^1\), the following identities hold in the sense of distributions:
\begin{align*}
&\rho_1 \frac{\mathrm{d}^2}{\mathrm{d}t^2}(\phi, w) + \kappa(\eta, w_x) - \kappa\left( \int_0^t g(t-s)\eta(s)\mathrm{d}s, w_x \right) = 0, \\
&\rho_2 \frac{\mathrm{d}^2}{\mathrm{d}t^2}(\psi, v) + b(\psi_x, v_x) + \kappa(\eta, v) - \kappa\left( \int_0^t g(t-s)\eta(s)\mathrm{d}s, v \right) = 0.
\end{align*}
That is, for every $\lambda \in C_c^{\infty}(0,T)$, one has
\begin{align*}
    &\int_0^T \rho_1(\phi,w)\lambda^{\prime\prime}(t)\mathrm{d}t=\int_0^T\left[ -\kappa(\eta,w_{x})+\kappa\left(\int_0^tg(t-s)\eta(s)\mathrm{d}s,w_{x}\right) \right] \lambda(t)\mathrm{d}t,\\
    &\int_0^T\rho_2(\psi,v)\lambda^{\prime\prime}(t)\mathrm{d}t=\int_0^T\left[ -b(\psi_{x},v_{x})-\kappa(\eta,v)+\kappa\left(\int_0^tg(t-s)\eta(s)\mathrm{d}s,v\right)\right]\lambda(t)\mathrm{d}t.
\end{align*}
Moreover, \(\phi\) satisfies the Dirichlet boundary conditions in \eqref{BC}, namely,
$\phi(0,t)=\phi(l,t)=0,\ t\in I_T$.
The Neumann boundary conditions for \(\psi\) in \eqref{BC} are understood in the weak sense through the term \((\psi_x,v_x)\) appearing above. The initial conditions are attained in the sense of \(L^2\) and \(L_*^2\), namely,
\begin{align}\label{weak_IC}
\left\{
\begin{aligned}
\phi(x,0)&=\phi_0(x),\quad \phi_t(x,0)=\phi_1(x),\\
\psi(x,0)&=\psi_0(x),\quad \psi_t(x,0)=\psi_1(x),
\end{aligned}
\right.
\quad \text{a.e. } x \in (0,l),
\end{align}
where $(\phi_0,\phi_1,\psi_0,\psi_1)\in H_0^1\times L^2 \times H_*^1 \times L_*^2$.
\end{definition}

\begin{definition}\label{str_sol}
We say that a pair of functions \((\phi,\psi)\) is a strong solution of the Timoshenko system \eqref{Timoshenko}--\eqref{IC} if it satisfies the regularity conditions
\begin{align*}
&\phi \in C(I_T; H^2 \cap H_0^1) \cap C^1(I_T; H_0^1) \cap C^2(I_T; L^2), \\
&\psi \in C(I_T; H^2 \cap H_*^1) \cap C^1(I_T; H_*^1) \cap C^2(I_T; L_*^2),
\end{align*}
and satisfies equation \eqref{Timoshenko} for all \(t \in I_T\) and for almost every \(x \in (0,l)\). In addition, the boundary conditions \eqref{BC} and the initial conditions \eqref{IC} are satisfied, where
$(\phi_0,\phi_1,\psi_0,\psi_1)\in(H^2\cap H_0^1)\times H_0^1\times(H^2\cap H_*^1)\times H_*^1$,
and \(\psi_0\) satisfies the compatibility conditions
$\psi_0^{(1)}(0)=\psi_0^{(1)}(l)=0$.
\end{definition}

We first state the well-posedness of the weak solution.
\begin{theorem}\label{thm:weak_sol}
Assume that \((\phi_0,\phi_1,\psi_0,\psi_1)\in H_0^1\times L^2 \times H_*^1 \times L_*^2\). Then the Timoshenko system \eqref{Timoshenko}--\eqref{IC} admits a unique weak solution \((\phi,\psi)\) and the following stability estimate holds:
\begin{equation*}
\begin{split}
&\| \phi \|_{C(I_T;H_0^1)} + \| \phi \|_{C^1(I_T;L^2)}
+\| \psi \|_{C(I_T;H_*^1)} + \| \psi \|_{C^1(I_T;L_*^2)}
\\
&\lesssim\;
\| \phi_0 \|_{H_0^1} + \| \phi_1 \|_{2}
+\| \psi_0 \|_{H_*^1} + \| \psi_1 \|_{2}.
\end{split}
\end{equation*}
\end{theorem}
\begin{proof}
The proof is carried out in four steps. 

\medskip
\noindent\textbf{Step 1. Construction of the Galerkin system}
\medskip

Let $w_j(x):=\sqrt{\frac{2}{l}}\mathrm{sin}(\frac{j\pi x}{l})$ and $v_j(x):=\sqrt{\frac{2}{l}}\mathrm{cos}(\frac{j\pi x}{l})$ for $x \in [0,l]$ and $j \in \mathbb{N}$. It can be checked that $\{w_j\}_{j=1}^{\infty}$ and $\{v_j\}_{j=1}^{\infty}$ are respectively the orthonormal basis of $H_0^1$ and $H_*^1$ (see \cite[Section 7.3.1]{Canuto2006}). We define the approximate solutions by
    \begin{equation}\label{as}
         \phi^m(x,t):=\sum_{j=1}^ma_{j}(t)w_j(x),\quad \psi^m(x,t):=\sum_{j=1}^mb_{j}(t)v_j(x),   
    \end{equation}
where $a_{j}$ and $b_{j}$, $j\in \mathbb{N}_m:=\left\{1,\dots,m\right\}$, are required to satisfy, for all $t\in (0,T]$,
    \begin{align}\label{gk1}
        \left.\left\{\begin{array}{l}\rho_1(\phi_{tt}^m,w_i)+\kappa(\eta^m,w_{i,x})-\kappa\left(\int_0^tg(t-s)\eta^m(s)\mathrm{d}s,w_{i,x}\right)=0,\\
      \rho_2(\psi_{tt}^m,v_i)+b(\psi_{x}^m,v_{i,x})+\kappa(\eta^m,v_{i})-\kappa\left(\int_0^tg(t-s)\eta^m(s)\mathrm{d}s,v_i\right)=0,\end{array}\right.\right. 
    \end{align}
together with the initial conditions
    \begin{align}\label{gkIC}
        \left.\left\{\begin{array}{l}a_{i}(0)=(\phi _0,w_i),\quad a^{\prime}_{i}(0)=(\phi _1,w_i),\\
        b_{i}(0)=(\psi _0,v_i),\quad b^{\prime}_{i}(0)=(\psi _1,v_i),  \end{array}\right.\right. 
    \end{align}
where $\eta^m:=\phi_x^m + \psi^m$, $w_{i,x}:=\frac{\mathrm{d}w_i}{\mathrm{d}x}$, $v_{i,x}:=\frac{\mathrm{d}v_i}{\mathrm{d}x}$, and $i\in \mathbb{N}_m$. Substituting the approximate solutions \eqref{as} into the Galerkin system \eqref{gk1}, we obtain
    \begin{equation}\label{ab_ode}
        \left.\left\{\begin{array}{l}
           a_{i}^{\prime \prime}(t)=-\kappa_1 c_i d_{i}(t)
          +\kappa_1  c_i \int_0^tg(t-s)d_{i}(s)\mathrm{d}s,
          \\
           b_{i}^{\prime \prime}(t)=-\kappa_2 d_{i}(t) -\bar{b} c_i^ 2 b_{i}(t)
          +\kappa_2 \int_0^tg(t-s)d_{i}(s)\mathrm{d}s,\end{array}\right.\right.    
    \end{equation}
where $\kappa_1:=\frac{\kappa}{\rho_1}$, $\kappa_2:=\frac{\kappa}{\rho_2}$, $\bar{b}:=\frac{b}{\rho_2}$, $c_i:=\frac{i \pi}{l}$ and $d_i:=c_i a_i +b_i$. For $i\in \mathbb{N}_m$, let $\boldsymbol{X}_i:=\left(a_{i}, b_{i}, a^{\prime}_{i}, b^{\prime}_{i}\right)^\top$.

Then the system of second-order Volterra integro-differential equations (VIDEs) \eqref{ab_ode} can be rewritten as the following first-order VIDE system:
    \begin{align}\label{ode}
        \frac{\mathrm{d}\boldsymbol{X}_i(t)}{\mathrm{d}t}=\boldsymbol{D}_i\boldsymbol{X}_i(t)+\int_0^tg(t-s)\boldsymbol{G}_i\boldsymbol{X}_i(s)\mathrm{d}s,\ t\in (0,T],\ \ i\in \mathbb{N}_m.
    \end{align}
    with
\[    
        \boldsymbol{D}_i:=\begin{pmatrix}0&0&1&0\\
        0&0&0&1\\
        -\kappa_1 c^2_{i}&-\kappa_1 c_{i}&0&0\\
        -\kappa_2 c_{i}&-\bar{b}c_i^2-\kappa_2&0&0\end{pmatrix},\quad
        \boldsymbol{G}_i:=\begin{pmatrix}0&0&0&0\\
        0&0&0&0\\
        \kappa_1 c^2_{i} &\kappa_1 c_{i}&0&0\\
        \kappa_2 c_{i}&\kappa_2&0&0\end{pmatrix},
\]
and the initial conditions \eqref{gkIC} for the Galerkin system yield the initial data for \eqref{ode}:
    \begin{align}\label{odeIC}
        \boldsymbol{X}_i(0)=\left((\phi _0,w_i),\ (\psi _0,v_i),\ (\phi _1,w_i),\ (\psi _1,v_i) \right)^{\top},\ i\in \mathbb{N}_m.
    \end{align}
Furthermore, \eqref{ode} can be transformed into the Volterra integral equation (VIE)
\begin{align}\label{e:V2}
        \boldsymbol{X}_i(t) =\boldsymbol{X}_i(0)
        +\int_0^t\boldsymbol{K}_i(t,s)\boldsymbol{X}_i(s)\mathrm{d}s,  
    \end{align}
where the kernel $\boldsymbol{K}_i(t,s):=\boldsymbol{D}_i+\int_s^t g(\tau-s)\mathrm{d}\tau \boldsymbol{G}_i$, and $\boldsymbol{K}_i \in C^2(D_T;\mathbb{R}^{4\times4})$ since $g \in C^1(I_T;\mathbb{R})$, with $D_T:=\{ (t,s)| 0\leq s \leq t \leq T \}$. From the theory of VIEs (see \cite[Theorem 2.1.7]{Bru04}), there exists a unique $\boldsymbol{X}_i\in C^2(I_T;\mathbb{R}^{4})$ satisfying \eqref{ode}--\eqref{odeIC}.

\medskip
\noindent \textbf{Step 2. A priori estimates}
\medskip

Let $m$ and $n$ be two positive integers with $m>n$. Set $G(t) := 1-\int_0^t g(s)\,\mathrm{d}s$. Define the difference functions 
$\Phi^{m,n}:=\phi^m-\phi^n,\ \Psi^{m,n}:=\psi^m-\psi^n$,
and introduce the corresponding energy
\[  
            E_1^{m,n}(t)
            := \frac{1}{2}\Bigl[
            \rho_1\|\Phi_t^{m,n}\|_2^2
            +\rho_2\|\Psi_t^{m,n}\|_2^2
            +b\|\Psi_x^{m,n}\|_2^2 
            \Bigr] 
            +\frac{\kappa}{2}\Bigl[
            G(t)\|H^{m,n}\|_2^2
            +(g\circ H^{m,n})(t)
            \Bigr]
 \]
with $H^{m,n}:=\Phi_x^{m,n}+\Psi^{m,n}$ and
    $$
        (g\circ H^{m,n})(t):=\int_0^l\int_0^tg(t-s)\left|H^{m,n}(t)-H^{m,n}(s)\right|^2\mathrm{d}s\mathrm{d}x.
    $$
Replacing $m$ by $n$ in \eqref{gk1} and subtracting the resulting system from \eqref{gk1}
yields
    \begin{align}\label{gk_Phsi}
        \left.\left\{\begin{array}{l}\rho_1(\Phi_{tt}^{m,n},w_i)+\kappa(H ^{m,n},w_{i,x})=\kappa\left(\int_0^tg(t-s)H ^{m,n}(s)\mathrm{d}s,w_{i,x}\right),\\
        \rho_2(\Psi_{tt}^{m,n},v_{i})+b(\Psi_{x}^{m,n},v_{i,x})+\kappa(H ^{m,n},v_{i})=\kappa\left(\int_0^tg(t-s)H ^{m,n}(s)\mathrm{d}s,v_{i}\right).\end{array}\right.\right. 
    \end{align}
Multiply the first equation in \eqref{gk_Phsi} by 
 $a^{\prime}_i$, the second by $b^{\prime}_i$, sum over
 $i = n+1, \dots,m$, then add the results;
the orthonormality of the $\{w_j\}_{j=1}^{\infty}$ and $\{v_j\}_{j=1}^{\infty}$  
gives
    \begin{align*}
        \frac{\mathrm{d}}{\mathrm{d}t}E_1^{m,n}(t)\leq -\frac{\kappa}{2}g(t)\|H^{m,n}\|_2^2+\frac{\kappa}{2}\int_0^l\int_0^tg'(t-s)|H^{m,n}(t)-H^{m,n}(s)|^2\mathrm{d}s\mathrm{d}x \leq 0,
    \end{align*}
here, we used the assumptions on \(g\) in \eqref{g} for the last inequality,
which implies that
    \begin{align}\label{e:Et0}
        E_1^{m,n}(t)\leq E_1^{m,n}(0).
    \end{align}
In addition, by the regularity of the initial data, we have
{\small\[
  \lim\limits_{n \to \infty}E_1^{m,n}(0) 
  = \lim\limits_{n \to \infty} \frac{1}{2} \Bigl[
   \rho_1\|\Phi_t^{m,n}(0)\|_2^2
   +\rho_2\|\Psi_t^{m,n}(0)\|_2^2
   +b\|\Psi_x^{m,n}(0)\|_2^2
   +\kappa\|H^{m,n}(0)\|_2^2
   \Bigr]=0;
\]}
thus, we obtain the following a priori estimate:
    \begin{align}\label{eg_esti}
    \lim\limits_{n \to \infty}E_1^{m,n}(t)\leq \lim\limits_{n \to \infty}E_1^{m,n}(0) 
    =0,\ \forall t\in I_T.
    \end{align}

\medskip
\noindent\textbf{Step 3. Construction of a weak solution and proof of uniqueness}
\medskip

By \eqref{eg_esti}, we obtain the estimates
\begin{align*}
&\| \Phi^{m,n} \|_{C(I_T;H^1_0)} =\sup_{t \in I_T}\|\Phi^{m,n}_x(t)\|_2 
\lesssim \sup_{t \in I_T}\|H^{m,n}(t)\|_2+ \sup_{t \in I_T}\|\Psi^{m,n}_x(t)\|_2\xrightarrow{n\to \infty}0,\\
&\| \Phi^{m,n} \|_{C^1(I_T;L^2)}
\lesssim  \sup_{t \in I_T}\|\Phi_x^{m,n}(t)\|_2+\sup_{t \in I_T}\|\Phi_t^{m,n}(t)\|_2\xrightarrow{n\to \infty}0,
\end{align*}
so $\{\phi^m\}$ is a Cauchy sequence both in $C(I_T;H^1_0)$ and $C^1(I_T;L^2)$, both of which are contained in $C(I_T;L^2)$, so there exists a unique  function $\phi\in C(I_T;H^1_0)\cap C^1(I_T;L^2)$ such that
$\phi^m \to \phi $.
Similarly, there exists a unique $\psi \in C(I_T;H^1_*)\cap C^1(I_T;L^2_*)$ such that
$\psi^m \to \psi$. In addition, by the Sobolev embedding
$H^1 \hookrightarrow C([0,l])$, we know that $(\phi,\psi)$ satisfies the initial and boundary conditions \eqref{weak_IC} and \eqref{BC}.

Let $W:=\{ w_i \}_{i=1}^\infty$ and $V:=\{ v_i \}_{i=1}^\infty$. For any $\lambda \in C_c^{\infty}(0,T)$, multiplying both sides of the Galerkin system \eqref{gk1} by $\lambda (t)$, integrating over $t\in I_T$, and then passing to the limit as $m\to \infty$, for any $w\in W$ and $v\in V$ we obtain 
\begin{align*}
	\int_0^T \rho_1(\phi,w)\lambda''(t)\,\mathrm{d}t
	&=
	\int_0^T \left[
	-\kappa(\eta,w_x)
	+\kappa\left(\int_0^t g(t-s)\eta(s)\,\mathrm{d}s,w_x\right)
	\right]\lambda(t)\,\mathrm{d}t, \\
	\int_0^T \rho_2(\psi,v)\lambda''(t)\,\mathrm{d}t
	&=
	\int_0^T \left[
	-b(\psi_x,v_x)
	-\kappa(\eta,v)
	+\kappa\left(\int_0^t g(t-                    s)\eta(s)\,\mathrm{d}s,v\right)
	\right]\lambda(t)\,\mathrm{d}t.
\end{align*}
Since $W$ and $V$ are dense in $H^1_0$ and $H^1_*$, respectively, the above identities hold for every $w\in H^1_0$ and $v \in H^1_*$. By Definition \ref{weak_sol}, $(\phi,\psi)$ is a weak solution of the system.

To prove the uniqueness, suppose that 
$\bar{\phi} \in C(0,T;H_0^1)$ $\cap$ $C^1(0,T;L^2)$ and
$\bar{\psi} \in C(0,T;H_*^1) \cap C^1(0,T;L_*^2)$,
and that for every $w \in H_0^1$ and $v \in H_*^1$, the following identities hold in the sense of distributions:
\begin{align}
	&\rho_1 \frac{\mathrm{d}^2}{\mathrm{d}t^2}(\bar{\phi}, w) + \kappa(\bar{\eta}, w_x) - \kappa\left( \int_0^t g(t-s)\bar{\eta}(s)\mathrm{d}s, w_x \right) = 0,\label{w_sol1}\\
	&\rho_2 \frac{\mathrm{d}^2}{\mathrm{d}t^2}(\bar{\psi}, v) + b(\bar{\psi}_x, v_x) + \kappa(\bar{\eta}, v) - \kappa\left( \int_0^t g(t-s)\bar{\eta}(s)\mathrm{d}s, v \right) = 0,\label{w_sol2}
\end{align}
where $\bar{\eta} := \bar{\phi}_x+\bar{\psi}$. Moreover, assume that $(\bar{\phi},\bar{\psi})$ satisfies the homogeneous initial conditions
$\bar{\phi}(x,0) = \bar{\phi}_t(x,0) = \bar{\psi}(x,0) = \bar{\psi}_t(x,0)=0,\ \text{a.e. } x \in (0,l)$,
and that $\bar{\phi}$ satisfies the homogeneous Dirichlet boundary conditions
$\bar{\phi}(0,t)=\bar{\phi}(l,t)=0,\ t\in I_T$.

Write
$\bar{\phi}(x,t)=\sum_{j=1}^{\infty}\bar{a}_j(t) w_j(x)$ and
$\bar{\psi}(x,t)=\sum_{j=1}^{\infty}\bar{b}_j(t) v_j(x)$,
for $\text{a.e.} \, x \in (0,l)$,
where $\bar{a}_i=(\bar{\phi},w_i),\bar{b}_i=(\bar{\psi},v_i) \in C^1(I_T)$. Substituting these expansions into \eqref{w_sol1} and \eqref{w_sol2}, and taking $w=w_i$ and $v=v_i$, similarly to Step 1, we obtain the VIE system \eqref{e:V2} with $\boldsymbol{X}_i:=\left(\bar{a}_{i},\bar{b}_{i},\bar{a}^{\prime}_{i}, \bar{b}^{\prime}_{i}\right)^\top$ and zero initial values $\boldsymbol{X}_i(0)=0$. Again by the theory of VIEs (see \cite[Theorem 2.1.7]{Bru04}), the unique solution $\boldsymbol{X}_i(t)\equiv0$. Hence $(\bar{\phi},\bar{\psi})=(0,0)$, which completes the proof of the uniqueness of weak solutions.

\medskip
\noindent\textbf{Step 4. Stability estimate}
\medskip

Define the energy functional as
\[
		E_1^m(t):= \frac{1}{2}\Bigl[
		\rho_1\|\phi_t^m\|_2^2
		+\rho_2\|\psi_t^m\|_2^2
		+b\|\psi_x^m\|_2^2
		\Bigr]
		+\frac{\kappa}{2}\Bigl[
		G(t)\|\eta^m\|_2^2
		+(g\circ\eta^m)(t)
		\Bigr],\ t\in I_T.
\]
Similarly to the derivative of the difference energy, we obtain
\begin{align*}
	\frac{\mathrm{d}}{\mathrm{d}t}E_1^m(t)=
	-\frac{\kappa}{2}g(t)\|\eta^m\|_2^2+\frac{\kappa}{2}\int_0^l\int_0^tg^{\prime}(t-s)|\eta^m(t)-\eta^m(s)|^2\mathrm{d}s\mathrm{d}x\leq 0.
\end{align*}
Consequently,
\begin{align}\label{eg_estimate}
	E_1^m(t)\leq E_1^m(0) \lesssim \left[\| \phi_0 \|_{H_0^1}+ \| \phi_1 \|_{2}+\| \psi_0 \|_{H_*^1}+ \| \psi_1 \|_{2}\right]^2,
\end{align}
which implies that
$$
\| \phi^m \|^2_{H_0^1}+  \| \phi^m_t \|^2_{2}+\| \psi^m \|^2_{H_*^1}+ \| \psi^m_t \|^2_{2} \lesssim  \left[\| \phi_0 \|_{H_0^1}+ \| \phi_1 \|_{2}+\| \psi_0 \|_{H_*^1}+ \| \psi_1 \|_{2}\right]^2,
$$
and the desired stability estimate follows by 
taking the limit $m \to \infty$.

This completes the proof of Theorem~\ref{thm:weak_sol}.
\end{proof}

The following second-order energy estimate for the Galerkin approximation is the key higher-regularity estimate in proving well-posedness of strong solutions and will be extended to arbitrary orders in Section~\ref{s3}.

\begin{lemma}\label{lem:second_order_energy}
Assume that $(\phi_0,\phi_1,\psi_0,\psi_1)\in (H^2\cap H_0^1)\times H_0^1\times (H^2\cap H_*^1)\times H_*^1$ with the compatibility condition $\psi^{(1)}_0(0)=\psi^{(1)}_0(l)=0$. Let \(m,n\in\mathbb{N}\) with \(m\ge n\). Define the second-order difference energy and the second-order energy by
\begin{align*}
        E_2^{m,n}(t):= &\frac{1}{2} \left[ \rho_1 \| \partial_t^{2} \Phi^{m,n} \|_2^2 + \rho_2 \| \partial_t^{2} \Psi^{m,n} \|_2^2+ b \| \partial_t \Psi_{x}^{m,n} \|_2^2+ \kappa \| \partial_t H^{m,n} \|_2^2 \right],\\
        E_2^{m}(t):= &\frac12\left[
        \rho_1\|\partial_t^2\phi^{m}(t)\|_2^2
        +\rho_2\|\partial_t^2\psi^{m}(t)\|_2^2
        +b\|\partial_t\psi_x^{m}(t)\|_2^2
        +\kappa\|\partial_t\eta^{m}(t)\|_2^2
        \right].
\end{align*}
 Then
\begin{align}
        &E_2^{m,n}(t)\lesssim \sum_{p=1}^2 E_p^{m,n}(0),
        \qquad
        \lim_{n\to\infty}E_2^{m,n}(0)=0,\label{e:E2mn}\\
       &E_2^m(t)\lesssim \sum_{p=1}^2 E_p^{m}(0)
        \lesssim
        \left[
        \|\phi_0\|_{H^2}+\|\psi_0\|_{H^2}
        +\|\phi_1\|_{H_0^1}+\|\psi_1\|_{H_*^1}
        \right]^2. \label{esti_e2m}
    \end{align}
Moreover, all mixed derivatives with total order $2$ for \((\Phi^{m,n},\Psi^{m,n})\) and \((\phi^{m},\psi^{m})\) satisfy
\begin{align}
        \sum_{k=0}^2\|\partial_t^k\partial_x^{2-k}\Phi^{m,n}(t)\|_2^2
        +\sum_{k=0}^2\|\partial_t^k\partial_x^{2-k}\Psi^{m,n}(t)\|_2^2
        \lesssim \sum_{p=1}^2 E_p^{m,n}(0),\label{eng_diff_equ2}\\
        \sum_{k=0}^2\|\partial_t^k\partial_x^{2-k}\phi^m(t)\|_2^2
        +\sum_{k=0}^2\|\partial_t^k\partial_x^{2-k}\psi^m(t)\|_2^2
        \lesssim \sum_{p=1}^2 E_p^{m}(0).\label{eng_equ2}
\end{align}
\end{lemma}
\begin{proof}
Differentiate \eqref{gk_Phsi} w.r.t. $t$, multiply the first equation by $a^{\prime \prime}_i$, the second by $b^{\prime \prime}_i$,  sum over $i=n+1, \dots, m$, then add the resulting identities;  we obtain
    \begin{align}\label{dE2}
        \frac{\mathrm{d}}{\mathrm{d}t}E_2^{m,n}(t)
        = -\kappa(
        \varrho^1_x,
        \Phi_{tt}^{m,n}) + 
        \kappa(\varrho^1,
        \Psi_{tt}^{m,n}),
    \end{align}
where $\varrho^1(t) := \partial_t \left( \int_0^t g(t-s) H^{m,n}(s)\mathrm{d}s \right) = g(0)H^{m,n}
+\int_0^t g'(t-s)H^{m,n}(s)\,\mathrm{d}s$.
Applying Cauchy-Schwarz, Minkowski and Young's inequalities yields 
    \begin{align}
       & \left|\kappa(
        \varrho^1_x,
        \Phi_{tt}^{m,n})\right| 
        \lesssim 
        \left\| H_x^{m,n}\right \|^2_2 + \int_0^t \left\| H_x^{m,n}(s) \right \|^2_2 \mathrm{d}s +\|\Phi_{tt}^{m,n}\|^2_2,\label{dE2_1}\\
    	&\left| \kappa(\varrho^1,
    	\Psi_{tt}^{m,n}) \right| \lesssim 
    	E_1^{m,n}(0)   +\|\Psi_{tt}^{m,n}\|^2_2,\label{dE2_2}
    \end{align}
where we used $\| H^{m,n} \|^2_2 \lesssim E_1^{m,n}(0)$ from \eqref{e:Et0}. 

Let $A^1:= \kappa \Phi_{xx}^{m,n}-\rho_1 \Phi_{tt}^{m,n}+\kappa \Psi_x^{m,n} - \kappa \int_0^tg(t-s)(\Phi_x^{m,n}+\Psi^{m,n})_x(s)\mathrm{d}s$ and $A^2:= b \Psi_{xx}^{m,n}- \rho_2 \Psi_{tt}^{m,n} - \kappa H^{m,n} + \kappa \int_0^tg(t-s)(\Phi_x^{m,n} + \Psi^{m,n})(s)\mathrm{d}s$. By \eqref{gk_Phsi}, 
we know that for all $i \in \mathbb{N}_m$,
$(A^1, w_i)=0$ and $(A^2, v_i)=0$.
Let $W_m:=\text{span}\{w_1, \dots , w_m\}$ and $V_m:=\text{span}\{v_1, \dots , v_m\}$.
But $A^1 \in W_m$ and $A^2 \in V_m$, so one must have $A^1 \equiv 0$ and $A^2 \equiv 0$,
 i.e., for all $t \in I_T$,
    \begin{align}\label{gk2}
        \left.\left\{\begin{array}{l}
        \kappa \Phi_{xx}^{m,n}= \rho_1 \Phi_{tt}^{m,n}-\kappa \Psi_x^{m,n} + \kappa \int_0^tg(t-s)(\Phi_x^{m,n}+\Psi^{m,n})_x(s)\mathrm{d}s,\\
        b \Psi_{xx}^{m,n}= \rho_2 \Psi_{tt}^{m,n} + \kappa H^{m,n} - \kappa \int_0^tg(t-s)(\Phi_x^{m,n} + \Psi^{m,n})(s)\mathrm{d}s.
        \end{array}\right.\right.   
    \end{align}
Again by \eqref{e:Et0}, and the Gronwall lemma \cite[Lemma 2.1.15]{Bru04}, we obtain
    \begin{align}
       &  \| \Phi^{m,n}_{xx} \|_2 \lesssim \| \Phi^{m,n}_{tt} \|_2 + \sqrt{E_1^{m,n}(0)} + \int_0^t \| \Phi_{tt}^{m,n}(s) \|_2 \mathrm{d}s,   \label{Phi_xx_Phi_tt}\\
        & \| \Psi^{m,n}_{xx} \|_2 \lesssim \| \Psi^{m,n}_{tt} \|_2 + \sqrt{E_1^{m,n}(0)}.   \label{Psi_xx_Psi_tt}
    \end{align}
Consequently
    \begin{align*}
        \left\| H_x^{m,n} \right\|^2_2 \leq \left\| \Phi_{xx}^{m,n} \right\|^2_2+E_1^{m,n}(0)
        \lesssim
        \| \Phi^{m,n}_{tt} \|^2_2 + E_1^{m,n}(0) + \int_0^t \| \Phi_{tt}^{m,n}(s) \|^2_2 \mathrm{d}s,
    \end{align*}
which together with \eqref{dE2}, \eqref{dE2_1} and \eqref{dE2_2} yields that
     \begin{align*}
        \frac{\mathrm{d}}{\mathrm{d}t}E_2^{m,n}(t) \lesssim 
        E_2^{m,n}(t) + \int_0^t E_2^{m,n}(s)\mathrm{d}s +E_1^{m,n}(0).  
    \end{align*}   
Furthermore, integrating this inequality gives
\[
   E_2^{m,n}(t) \lesssim \sum_{p=1}^2 E_p^{m,n}(0) + \int_0^t E_2^{m,n}(s)\mathrm{d}s,
\] 
which together with the Gronwall lemma \cite[Lemma 2.1.15]{Bru04} yields the first inequality of \eqref{e:E2mn}, and the second one follows by the regularity of the initial data together with the compatibility condition.
Moreover, \eqref{eng_diff_equ2} follows from \eqref{Phi_xx_Phi_tt}, \eqref{e:E2mn} and \eqref{Psi_xx_Psi_tt}.

Similarly, we have the estimates
$E_2^{m}(t) \lesssim E_2^{m}(0) + E_1^{m}(0)$, $\| \phi^{m}_{xx} \|_2 \lesssim \| \phi^{m}_{tt} \|_2 + \sqrt{E_1^{m}(0)} + \int_0^t \| \phi_{tt}^{m}(s) \|_2 \mathrm{d}s$
and
$\| \psi^{m}_{xx} \|_2 \lesssim \| \psi^{m}_{tt} \|_2 +\sqrt{E_1^{m}(0)}$; and
    \begin{align}\label{gk3}
          \left.\left\{\begin{array}{l}
            \kappa \phi_{xx}^{m}= \rho_1 \phi_{tt}^{m}-\kappa \psi_x^{m} + \kappa \int_0^tg(t-s)(\phi_x^{m}+\psi^{m})_x(s)\mathrm{d}s ,\\
            b \psi_{xx}^{m}= \rho_2 \psi_{tt}^{m} + \kappa (\phi_x^{m}+\psi^{m}) - \kappa \int_0^tg(t-s)(\phi_x^{m} + \psi^{m})(s)\mathrm{d}s,
          \end{array}\right.\right.  \forall t\in I_T. 
    \end{align}
Thus, \eqref{eng_equ2} follows. In addition,
by the regularity of the initial data together with the compatibility conditions, 
    \begin{align*}
         \| \phi_{xx}^m(0) \|_2 \leq \| \phi_{0} \|_{H^2} ,\ \| \phi_{tt}^m(0) \|_2 \lesssim \| \phi_{0} \|_{H^2} + \| \psi_{0} \|_{H^2} ,\  \| \psi_{xx}^m(0) \|_2 \leq \| \psi_{0} \|_{H^2},
         \\
        \| \psi_{tt}^m(0) \|_2 \lesssim \| \phi_{0} \|_{H^2} + \| \psi_{0} \|_{H^2} ,\ \| \phi_{xt}^m(0) \|_2 \leq \| \phi_{1} \|_{H_0^1} ,\ \| \psi_{xt}^m(0) \|_2 \le \| \psi_{1} \|_{H_*^1},
    \end{align*}
which together with the estimate of \(E_1^m(0)\) in \eqref{eg_estimate}, yield \eqref{esti_e2m}.
\end{proof}

Now we state the strong well-posedness result.
\begin{theorem}\label{thm:strong_sol}
Assume that \((\phi_0,\phi_1,\psi_0,\psi_1)\in(H^2\cap H_0^1)\times H_0^1\times(H^2\cap H_*^1)\times H_*^1\) with the compatibility condition $\psi^{(1)}_0(0)=\psi^{(1)}_0(l)=0$.
Then the Timoshenko system \eqref{Timoshenko}--\eqref{IC} admits a unique strong solution \((\phi,\psi)\). Moreover,
\begin{align*}
\sum_{k=0}^2 \| \phi \|_{C^k(I_T;H^{2-k}\cap H_0^1)}
+\sum_{k=0}^2 \| \psi \|_{C^k(I_T;H^{2-k}\cap H_*^1)}
\\
\lesssim\;
\| \phi_{0} \|_{H^2} + \| \psi_{0} \|_{H^2}
+ \| \phi_{1} \|_{H_0^1} + \| \psi_{1} \|_{H_*^1},
\end{align*}
where $H^{0}\cap H_0^1:=L^2$ and $H^{0}\cap H_*^1:=L_*^2$.
\end{theorem}
\begin{proof}
	By the second-order energy estimate in Lemma  \ref{lem:second_order_energy}, the sequences $\{\phi^m\}$, $\{\psi^m\}$ are uniformly bounded and Cauchy respectively in the spaces $C(I_T; H^2\cap H_0^1) \cap C^1(I_T; H_0^1) \cap C^2(I_T; L^2)$
	and $C(I_T;H_*^1\cap H^2)\cap C^1(I_T;H_*^1)\cap C^2(I_T;L_*^2)$.	
Then, there exist limits $\phi$, $\psi$ belonging to the above spaces. Letting $m\to\infty$ in the Galerkin system \eqref{gk3} and using the convergence, we obtain that $(\phi,\psi)$ satisfies the original equations \eqref{Timoshenko}, the boundary conditions \eqref{BC} and the initial conditions \eqref{IC}. Hence $(\phi,\psi)$ is a strong solution.
Since every strong solution is also a weak solution, and the weak solution is unique by Theorem \ref{thm:weak_sol}, the strong solution is also unique. The stability estimate follows directly from Lemma \ref{lem:second_order_energy}.
\end{proof}

\section{Higher-order regularity}\label{s3}
In this section, we study the higher-order regularity of solutions to the Timoshenko system \eqref{Timoshenko}--\eqref{IC}. The key step is to extend the second-order energy estimate established in Lemma~\ref{lem:second_order_energy} to arbitrary orders \(q\ge 2\). We first derive the corresponding higher-order a priori estimates for the Galerkin approximations and then use them to obtain regularity results for strong solutions. 

The following lemma provides the required higher-order energy estimate.
\begin{lemma}\label{qth_energy}
Let \(q\in\mathbb{N}\) with \(q\geq 2\). Assume that
$(\phi_0,\phi_1,\psi_0,\psi_1)\in(H^q\cap H_0^1)\times (H^{q-1}\cap H_0^1)\times(H^q\cap H_*^1)\times (H^{q-1}\cap H_*^1)$, \(g\in C^{q-1}(I_T;\mathbb{R})\), and the initial data satisfy the following compatibility conditions at \(x=0,\ l\):
    \begin{itemize}
        \item If \(q=2r\) is even, 
            \begin{align}
                &\phi_0^{(2j)}(x)=0\quad (1\le j\le r-1),\quad
                \psi_0^{(2j-1)}(x)=0\quad (1\le j\le r),\label{compatibility_even}\\
                &\phi_1^{(2j)}(x)=0\quad (1\le j\le r-1),\quad
                \psi_1^{(2j-1)}(x)=0\quad (1\le j\le r-1).\label{compatibility_even_1}
            \end{align}
        \item If \(q=2r+1\) is odd,
            \begin{align}
                &\phi_0^{(2j)}(x)=0\quad (1\le j\le r),\quad
                \psi_0^{(2j-1)}(x)=0\quad (1\le j\le r),\label{compatibility_odd}\\
                &\phi_1^{(2j)}(x)=0\quad (1\le j\le r-1),\quad
                \psi_1^{(2j-1)}(x)=0\quad (1\le j\le r).\label{compatibility_odd_1}
            \end{align}
    \end{itemize}
These compatibility conditions ensure that the high-order spatial derivatives of the initial data are compatible with the sine expansion for \(\phi\) and the cosine expansion for \(\psi\). In particular, they allow repeated integration by parts without boundary contributions when identifying Sobolev norms with weighted Fourier coefficients.
Define the $q$-th order difference energy and the $q$-th order energy by
\begin{align*}
       & E_q^{m,n}(t):=\frac{1}{2} \left[ \rho_1 \| \partial_t^{q} \Phi^{m,n} \|_2^2 + \rho_2 \| \partial_t^{q} \Psi^{m,n} \|_2^2+ b \| \partial_t^{q-1} \Psi_{x}^{m,n} \|_2^2+ \kappa \| \partial_t^{q-1} H^{m,n} \|_2^2 \right],\\
      &  E_q^{m}(t):=\frac{1}{2} \left[ \rho_1 \| \partial_t^{q} \phi^{m} \|_2^2 + \rho_2 \| \partial_t^{q} \psi^{m} \|_2^2+ b \| \partial_t^{q-1} \psi_{x}^{m} \|_2^2+ \kappa \| \partial_t^{q-1} \eta ^{m} \|_2^2 \right],
  \end{align*}
where \(m,n\in\mathbb{N}\) and \(m\geq n\). Let $\mathbf{E}_q^{m,n}:=\sum_{p=1}^{q} E_p^{m,n}(0)$.
Then we have
    \[
        E_q^{m,n}(t) \lesssim \mathbf{E}_q^{m,n},\quad
        \lim_{n\to\infty}E_p^{m,n}(0)=0,\quad \forall p\in\{2,\dots,q\},
    \]
and
    \[
        \mathbf{E}_q^m=\sum_{p=1}^{q} E_p^{m}(0)
        \lesssim
        \left[ \| \phi_{0} \|_{H^q} + \| \psi_{0} \|_{H^q} + \| \phi_{1} \|_{H^{q-1}} + \| \psi_{1} \|_{H^{q-1}} \right]^2.
    \]
Moreover, all mixed derivatives with total order \(q\) of \((\Phi^{m,n},\Psi^{m,n})\) and \((\phi^{m},\psi^{m})\) satisfy
    \begin{align}
       & \sum_{k=0}^q \| \partial_t^k \partial_x^{q-k} \Phi^{m,n} \|_2^2
        +\sum_{k=0}^q \| \partial_t^k \partial_x^{q-k} \Psi^{m,n} \|_2^2
        \lesssim \mathbf{E}_q^{m,n},\label{eng_diff_equ}\\
        &\sum_{k=0}^q \| \partial_t^k \partial_x^{q-k} \phi^{m} \|_2^2
        +\sum_{k=0}^q \| \partial_t^k \partial_x^{q-k} \psi^{m} \|_2^2
        \lesssim \mathbf{E}_q^{m}.\label{eng_equ}
    \end{align}
\end{lemma}
\begin{proof}
	We use induction. The case $q=2$ was proved already in Lemma \ref{lem:second_order_energy}. Suppose that the result is valid for some $q \in \mathbb{N}$ with $q \geq 2$. We will prove it is also valid for the case $q+1$. The proof is divided into five steps.
	
	\medskip
	\noindent\textbf{Step 1. Construction of the Galerkin system}
	\medskip
	
	By Step 1 of the proof for Theorem \ref{thm:weak_sol}, we obtain the VIE \eqref{e:V2} with $\boldsymbol{K}_i \in C^{q+1}(D_T;\mathbb{R}^{4 \times 4})$
	since $g \in C^{q}(I_T;\mathbb{R})$, which has a unique solution $\boldsymbol{X}_i\in C^{q+1}(I_T;\mathbb{R}^{4})$ (see \cite[Theorem 2.1.7]{Bru04}). In particular, $a_i,\ b_i \in C^{q+2}(I_T)$ for all $i \in \mathbb{N}_m$.
	
	\medskip
	\noindent\textbf{Step 2. Estimates for $E_{q+1}^{m,n}(t)$ and $E_{q+1}^{m,n}(0)$}
	\medskip
	
	Differentiate \eqref{gk_Phsi} $q$ times w.r.t. $t$, multiply the first equation by $a^{(q+1)}_i$ and the second by $b^{(q+1)}_i$, sum over $i = n+1, \dots, m$, and add the two resulting identities, then
	\begin{align}\label{dEq}
		\frac{\mathrm{d}}{\mathrm{d}t}E_{q+1}^{m,n}(t)  
		=-\kappa\left(\varrho^q_x(t),\partial_t^{q+1}\Phi^{m,n}\right)+\kappa\left( \varrho^q(t),\partial_t^{q+1}\Psi^{m,n}\right),
	\end{align}
where
	$\varrho^q(t) :=\partial_t^q \left( \int_0^t g(t-s) H^{m,n}(s)\mathrm{d}s \right) = \sum_{k=0}^{q-1}g^{(k)}(0)\partial_t^{q-1-k}H^{m,n}(t)+$ $\int_0^t g^{(q)}(t - s) H^{m,n}(s) \mathrm{d}s$.
By the Cauchy-Schwarz, Minkowski and Young's inequalities, 
and the induction hypothesis \eqref{eng_diff_equ}, we obtain
	\begin{align}
	&	\left|\kappa\left(\varrho^q_x(t),\partial_t^{q+1}\Phi^{m,n}\right)\right|
		\lesssim 
		\left\| g(0) \partial_t^{q-1}H_x^{m,n} \right \|^2_2 + \mathbf{E}^{m,n}_q    +\| \partial_t^{q+1}\Phi^{m,n} \|^2_2, \label{dEq_1}\\
	&	\left| \kappa \left(\varrho^q(t),\partial_t^{q+1} \Psi^{m,n}\right) \right| \lesssim 
		\mathbf{E}_q^{m,n}   +\|\partial_t^{q+1} \Psi^{m,n}\|^2_2.\label{dEq_2}
	\end{align}
Differentiate the first identity in \eqref{gk2} $q-1$ times w.r.t. $t$ to get
\begin{align}\label{q_Phi_xx_eq_Phi_tt}
	\kappa \partial_t^{q-1} \Phi_{xx}^{m,n}=& \rho_1 \partial_t^{q+1}\Phi^{m,n}-\kappa \partial_t^{q-1} \Psi_x^{m,n} +  \kappa \varrho_x^{q-1}.        
\end{align}
Using the induction hypothesis \eqref{eng_diff_equ}, it follows that
	\begin{align*}
		\| \partial_t^{q-1} \Phi^{m,n}_{xx} \|^2_2 \lesssim \| \partial_t^{q+1}\Phi^{m,n} \|^2_2 + \mathbf{E}_q^{m,n},  
	\end{align*}
which together with \eqref{dEq}, \eqref{dEq_1} and \eqref{dEq_2} yields
	\begin{align*}
		\frac{\mathrm{d}}{\mathrm{d}t}E_{q+1}^{m,n}(t)    
		\lesssim E_{q+1}^{m,n}(t) + \mathbf{E}_{q}^{m,n}.
	\end{align*}
Integrating and applying Gronwall's lemma \cite[Lemma 2.1.15]{Bru04} gives
	\begin{align*}
		E_{q+1}^{m,n}(t) 
		\lesssim E_{q+1}^{m,n} (0)
		+ \mathbf{E}_{q}^{m,n} = \mathbf{E}_{q+1}^{m,n}.
	\end{align*}
	We next prove $\lim\limits_{n \to \infty} E_{q+1}^{m,n}(0)=0$. 
Differentiating \eqref{ab_ode} $k$ times w.r.t. $t$ with $k \in \mathbb{N}$ 
then evaluating at $t=0$, one gets	
	\begin{align}
	&	a^{(k+2)}_j(0)=-\kappa_1 c_j d^{(k)}_{j}(0)
		+\kappa_1  c_j \sum _{p=0}^{k-1} g^{(p)}(0) d_j^{(k-1-p)}(0),\label{ak}\\
	&	b^{(k+2)}_j(0)=- \kappa_2  d^{(k)}_{j}(0)- \bar{b} c^2_j b^{(k)}_{j}(0)
		+\kappa_2  \sum _{p=0}^{k-1} g^{(p)}(0) d_j^{(k-1-p)}(0).   \label{bk} 
	\end{align}
Observe that when $k=1$, the right-hand sides of \eqref{ak} and \eqref{bk} are linear combinations of
	$a_j(0)$, $a^{\prime}_j(0)$, $b_j(0)$ and $b^{\prime}_j(0)$,
	with coefficients that are polynomials in $c_j$. 
	Iterating \eqref{ak} and \eqref{bk}, one sees that $a^{(q-1)}_j(0)$ is also such a combination, and
	\begin{equation*}
		\begin{split}
			\partial_t^{q-1}\Phi_{xx}^{m,n}(0)
			=& \sum_{j=n+1}^m c_j^2 a_j^{(q-1)}(0) w_j= \sum_{j=n+1}^m c_j^2 P_{q-1}(c_j)a_j(0)w_j + c_j^2 Q_{q-2}(c_j)a'_j(0)w_j \\
			&+\sum_{j=n+1}^m c_j^2 R_{q-1}(c_j)b_j(0)w_j + c_j^2 S_{q-2}(c_j)b'_j(0)w_j;
		\end{split}
	\end{equation*}
here, if $q$ is odd, then $P_{q-1}$ and $R_{q-1}$ are polynomials of degree $q-1$, while $Q_{q-2}$ and $S_{q-2}$ are polynomials of degree $q-3$; if $q$ is even, then $P_{q-1}$, $R_{q-1}$, $Q_{q-2}$ and $S_{q-2}$ are polynomials of degree $q-2$. Thus, we now consider these two cases separately.
	
	If $q$ is odd, the assumptions $(\phi_0,\phi_1,\psi_0,\psi_1)\in(H^{q+1}\cap H_0^1)\times (H^{q}\cap H_0^1)\times(H^{q+1}\cap H_*^1)\times (H^{q}\cap H_*^1)$,
	and the $q+1$-th order compatibility conditions \eqref{compatibility_even}--\eqref{compatibility_even_1} yield
	\begin{align*}
		&\left\| \phi_0^{(q+1)} \right\|^2_2=\Big\| \sum_{j=1}^{\infty} \left( \phi _0^{(q+1)}, w_j \right) w_j \Big\|^2_2 =  \sum_{j=1}^{\infty} [c^{q+1}_j \left(  \phi _0, w_j \right)]^2
		=  \sum_{j=1}^{\infty} [c^{q+1}_j a_j(0)]^2 <+\infty;
	\end{align*}
	similarly,
$\left\| \psi_0^{(q+1)} \right\|^2_2=  \sum_{j=1}^{\infty} [c^{q+1}_j b_j(0)]^2<+\infty$, $\left\|  \phi_1^{(q-1)} \right\|^2_2=  \sum_{j=1}^{\infty} [c^{q-1}_j a'_j(0)]^2<+\infty$ and $\left\|  \psi_1^{(q-1)} \right\|^2_2=  \sum_{j=1}^{\infty} [c^{q-1}_j b'_j(0)]^2<+\infty$.

	If $q$ is even, then 
	$(\phi_0,\phi_1,\psi_0,\psi_1)$ satisfies the $q+1$-th order compatibility conditions \eqref{compatibility_odd}--\eqref{compatibility_odd_1}. Thus,
$\left\|  \phi_0^{(q)} \right\|^2_2=  \sum_{j=1}^{\infty} [c^{q}_j a_j(0)]^2<+\infty$, $\left\| \psi_0^{(q)} \right\|^2_2=  \sum_{j=1}^{\infty} [c^{q}_j b_j(0)]^2<+\infty$, and $\left\| \phi_1^{(q)} \right\|^2_2=  \sum_{j=1}^{\infty} [c^{q}_j a'_j(0)]^2<+\infty,\quad \left\| \psi_1^{(q)} \right\|^2_2=  \sum_{j=1}^{\infty} [c^{q}_j b'_j(0)]^2<+\infty$.

	Hence, in both cases
	$\lim\limits_{n \to \infty} \| \partial_t^{q-1} \Phi^{m,n} _{xx}(0) \|^2_2=0$. 
Similarly,
	$$
	\lim\limits_{n \to \infty} \| \partial_t^{q-1} \Psi^{m,n} _{xx}(0) \|^2_2
	=
	\lim\limits_{n \to \infty} \| \partial_t^{q} \Phi^{m,n} _{x}(0) \|^2_2
	=
	\lim\limits_{n \to \infty} \| \partial_t^{q} \Psi^{m,n} _{x}(0) \|^2_2
	=0.
	$$
	Setting $t=0$ in \eqref{q_Phi_xx_eq_Phi_tt}, we infer that
	$\lim\limits_{n \to \infty} \| \partial_t^{q+1} \Phi^{m,n} (0) \|^2_2=0$.
Similarly, differentiate the second equation in \eqref{gk2} $q-1$ times w.r.t. $t$,
then set $t=0$;  one obtains
	$\lim\limits_{n \to \infty} \| \partial_t^{q+1} \Psi^{m,n} (0) \|^2_2=0$.
	Consequently, we have proved that
	$\lim\limits_{n \to \infty} \mathbf{E}_{q+1}^{m,n}=0$.
	
	\medskip
	\noindent\textbf{Step 3. Showing that all derivatives of order \(q+1\) for $\Phi^{m,n}$ and $\Psi^{m,n}$ are bounded by $\mathbf{E}_{q+1}^{m,n}$}
	\medskip
	
	We shall prove the following two inequalities by induction:
	$$
	\| \partial_t^{q+1-k} \partial_x^k \Phi^{m,n} \|^2_2  \lesssim \mathbf{E}_{q+1}^{m,n} ,\quad \| \partial_t^{q+1-k} \partial_x^k \Psi^{m,n} \|^2_2 \lesssim \mathbf{E}_{q+1}^{m,n} \ \forall k \in \{ 0,\dots , q+1 \}.
	$$
	The cases $k=0,1,2$ have already been established in Step 2. Assume now that the estimate holds for all integers up to $k$, and consider the case $k+1$. Applying the operator $\partial^{q-k}_t \partial^{k-1}_x$ to the first equation in \eqref{gk2}, we obtain
		\begin{align*}
		\kappa \partial_t^{q-k}\partial_x^{k+1}\Phi^{m,n}
		= \rho_1 \partial_t^{q+1-(k-1)}\partial_x^{k-1}\Phi^{m,n}
		-\kappa \partial_t^{q-k}\partial_x^k \Psi^{m,n} 
		+ \partial_x^{k-1}\varrho^{q-k}.
	\end{align*}
	By the induction hypothesis, if $k \leq q-1$, then
	$\| \partial_t^{q-k} \partial_x^{k+1} \Phi^{m,n} \|^2_2  \lesssim \mathbf{E}_{q+1}^{m,n}$;
	if $k = q$, then a Gronwall inequality \cite[Lemma 2.1.15]{Bru04} gives
	$\| \partial_x^{q+1} \Phi^{m,n} \|^2_2  \lesssim \mathbf{E}_{q+1}^{m,n}$.
	Similarly, applying the operator $\partial^{q-k}_t \partial^{k-1}_x$ to the second equation in \eqref{gk2} and using the induction hypothesis, we obtain
	$\| \partial_t^{q-k} \partial_x^{k+1} \Psi^{m,n} \|^2_2  \lesssim \mathbf{E}_{q+1}^{m,n} ,\ \forall k \in \{1, \dots, q \}$.
	
	\medskip
	\noindent\textbf{Step 4. Estimates for $E_{q+1}^{m}(t)$ and $E_{q+1}^{m}(0)$}
	\medskip
	
Similarly to Step 2, one has
	$E^m_{q+1}(t) \lesssim \mathbf{E}_{q+1}^m$.
	By the induction hypothesis,
	$$
	\mathbf{E}_q^m \lesssim \left[ \| \phi_{0} \|_{H^{q+1}} + \| \psi_{0} \|_{H^{q+1}} + \| \phi_{1} \|_{H^{q}} + \| \psi_{1} \|_{H^{q}} \right]^2;
	$$
	therefore, it remains to prove that
	$$
	E^m_{q+1}(0) \lesssim \left[\| \phi_{0} \|_{H^{q+1}} + \| \psi_{0} \|_{H^{q+1}} + \| \phi_{1} \|_{H^{q}} + \| \psi_{1} \|_{H^{q}} \right]^2.
	$$
	From the analysis of $ \partial_t^{q-1} \Phi^{m,n} _{xx}(0) $ in Step 2, we know that
	\begin{align*}
		\partial_t^{q-1} \phi^{m} _{xx}(0) &= \sum_{j=1}^m c_j^2 P_{q-1}(c_j) a_j(0) w_j + c_j^2 Q_{q-2}(c_j)  a^{\prime}_j(0) w_j \\
		&+ \sum_{j=1}^m c_j^2 R_{q-1}(c_j) b_j(0) w_j + c_j^2 S_{q-2}(c_j) b^{\prime}_j(0) w_j.
	\end{align*}
	If $q$ is odd, then the $q+1$-th order compatibility conditions \eqref{compatibility_even}--\eqref{compatibility_even_1} imply that
\begin{align*}
	\Big\|  \sum_{j=1}^m c_j^{k} a_j(0) w_j \Big\|^2_2  \leq \|  \phi_0^{(k)} \|^2_2 ,\quad
	\Big\|  \sum_{j=1}^m c_j^{k} b_j(0) v_j \Big\|^2_2  \leq \|  \psi_0^{(k)} \|^2_2 ,\quad 0 \leq k \leq q+1,\\
	\Big\|  \sum_{j=1}^m c_j^{k} a^{\prime}_j(0) w_j \Big\|^2_2  \leq \|  \phi_1^{(k)} \|^2_2 ,\quad
	\Big\|  \sum_{j=1}^m c_j^{k} b^{\prime}_j(0) v_j \Big\|^2_2  \leq \|  \psi_1^{(k)} \|^2_2 ,\quad 0 \leq k \leq q-1;
	\end{align*}
hence,
	\begin{align*}
		\| \partial_t^{q-1} \phi^{m} _{xx}(0) \|_2^2 &\lesssim \sum_{k=0}^{q+1} \| \phi _0^{(k)} \|^2_2 + \sum_{k=0}^{q+1} \| \psi_0^{(k)} \|^2_2 + \sum_{k=0}^{q-1} \| \phi_1^{(k)} \|^2_2 + \sum_{k=0}^{q-1} \| \psi_1^{(k)} \|^2_2\\
		&\lesssim  \left[ \| \phi_{0} \|_{H^{q+1}} + \| \psi_{0} \|_{H^{q+1}} + \| \phi_{1} \|_{H^{q}} + \| \psi_{1} \|_{H^{q}} \right]^2.
	\end{align*}
	If $q$ is even, then we get similarly
	\begin{align*}
		\| \partial_t^{q-1} \phi^{m} _{xx}(0) \|_2^2 &\lesssim \sum_{k=0}^{q} \| \phi _0^{(k)} \|^2_2 + \sum_{k=0}^{q} \| \psi_0^{(k)} \|^2_2 + \sum_{k=0}^{q} \| \phi_1^{(k)} \|^2_2 + \sum_{k=0}^{q} \| \psi_1^{(k)} \|^2_2\\
		&\lesssim  \left[ \| \phi_{0} \|_{H^{q+1}} + \| \psi_{0} \|_{H^{q+1}} + \| \phi_{1} \|_{H^{q}} + \| \psi_{1} \|_{H^{q}} \right]^2.
	\end{align*}
	Applying the same argument to $\partial_t^{q-1} \psi^{m,n} _{xx}(0)$, $\partial_t^{q} \phi_x^{m,n} (0)$ 
	and $\partial_t^{q} \psi^{m,n} _{x}(0)$ gives
	\begin{align*}
		\| \partial_t^{q-1} \psi^{m} _{xx}(0) \|_2^2 \lesssim  \left[ \| \phi_{0} \|_{H^{q+1}} + \| \psi_{0} \|_{H^{q+1}} + \| \phi_{1} \|_{H^{q}} + \| \psi_{1} \|_{H^{q}} \right]^2,\\
		\| \partial_t^{q} \phi^{m} _{x}(0) \|_2^2 \lesssim  \left[ \| \phi_{0} \|_{H^{q+1}} + \| \psi_{0} \|_{H^{q+1}} + \| \phi_{1} \|_{H^{q}} + \| \psi_{1} \|_{H^{q}} \right]^2,\\
		\| \partial_t^{q} \psi^{m} _{x}(0) \|_2^2 \lesssim  \left[ \| \phi_{0} \|_{H^{q+1}} + \| \psi_{0} \|_{H^{q+1}} + \| \phi_{1} \|_{H^{q}} + \| \psi_{1} \|_{H^{q}} \right]^2.
	\end{align*}
	Applying the operator $\partial^{q-1}_t$ to both equations in \eqref{gk3} and then setting $t=0$, and combining this with the induction hypothesis \eqref{eng_equ} as well as the above estimates for
	$\|   \partial_t^{q-1} \phi^{m} _{xx}(0) \|_2^2 $
	and
	$\|   \partial_t^{q-1} \psi^{m} _{xx}(0) \|_2^2 $,
	we obtain
	$$
	\| \partial_t^{q+1}\phi^{m}(0) \|_2^2 + \| \partial_t^{q+1}\psi^{m}(0) \|_2^2 \lesssim \left[ \| \phi_{0} \|_{H^{q+1}} + \| \psi_{0} \|_{H^{q+1}} + \| \phi_{1} \|_{H^{q}} + \| \psi_{1} \|_{H^{q}} \right]^2,
	$$
	which together with the estimates for $\|   \partial_t^{q} \phi_x^{m} (0) \|_2^2$ and $\|   \partial_t^{q} \psi_x^{m} (0) \|_2^2$, yields
	$$
	E^m_{q+1}(0) \lesssim \left[ \| \phi_{0} \|_{H^{q+1}} + \| \psi_{0} \|_{H^{q+1}} + \| \phi_{1} \|_{H^{q}} + \| \psi_{1} \|_{H^{q}} \right]^2.
	$$
	
	\medskip
	\noindent\textbf{Step 5. Showing that all derivatives of order \(q+1\) for $\phi^{m}$ and $\psi^{m}$ are bounded by $\mathbf{E}_{q+1}^{m}$}
	\medskip
	
	The proof is entirely analogous to Step 3, with $(\Phi^{m,n},\Psi^{m,n})$ replaced by $(\phi^{m},\psi^{m})$ and $\mathbf{E}_{q+1}^{m,n}$ replaced by $\mathbf{E}_{q+1}^{m}$. Using induction together with \eqref{gk3}, we obtain
	$$
	\| \partial_t^{q+1-k} \partial_x^k \phi^{m} \|^2_2  \lesssim \mathbf{E}_{q+1}^{m} ,\ \| \partial_t^{q+1-k} \partial_x^k \psi^{m} \|^2_2 \lesssim \mathbf{E}_{q+1}^{m} ,\ \forall k \in \{ 0,\dots , q+1 \}.
	$$
	This completes the induction and hence the proof of the lemma.
\end{proof}

The higher-order energy estimate of the previous lemma yields the following regularity theorem for strong solutions.
\begin{theorem}\label{regularity}
	Let \(q\in\mathbb{N}\) with \(q\geq 2\). Assume that
	\((\phi_0,\phi_1,\psi_0,\psi_1)\in(H^q\cap H_0^1)\times (H^{q-1}\cap H_0^1)\times(H^q\cap H_*^1)\times (H^{q-1}\cap H_*^1)\), \(g\in C^{q-1}(I_T)\), and that the initial data satisfy the compatibility conditions \eqref{compatibility_even}--\eqref{compatibility_even_1} or \eqref{compatibility_odd}--\eqref{compatibility_odd_1}. Then the Timoshenko system \eqref{Timoshenko}--\eqref{IC} admits a unique strong solution \((\phi,\psi)\) such that
	\[
	\phi\in \bigcap_{k=0}^q C^k\left(I_T;H^{q-k}\cap H_0^1\right),\quad
	\psi\in \bigcap_{k=0}^q C^k\left(I_T;H^{q-k}\cap H_*^1\right),
	\]
	and the following stability estimate holds:
	\begin{align*}
		&\sum_{k=0}^q \| \phi \|_{C^k\left(I_T;H^{q-k}\cap H_0^1\right)}
		+\sum_{k=0}^q \| \psi \|_{C^k\left(I_T;H^{q-k}\cap H_*^1\right)}
		\\
		&\lesssim
		\| \phi_0 \|_{H^q}+ \| \phi_1 \|_{H^{q-1}}+\| \psi_0 \|_{H^q}+ \| \psi_1 \|_{H^{q-1}}.
	\end{align*}
\end{theorem}
\begin{proof}
We proceed by induction on $q$. The case $q=2$ is exactly Theorem \ref{thm:strong_sol}. Assume that the theorem holds for some $q\ge 2$ and consider the case $q+1$. 
The higher-order energy estimates in Lemma \ref{qth_energy} yield  $E_{q+1}^{m,n}(t)\lesssim \mathbf{E}_{q+1}^{m,n}$ and $\lim_{n\to\infty}\mathbf{E}_{q+1}^{m,n}=0$; consequently, all mixed derivatives $\partial_t^k\partial_x^{q+1-k}$ are bounded in the $L^2$ norm by $\mathbf{E}_{q+1}^{m,n}$. Therefore $\{\phi^m\}$ and $\{\psi^m\}$ are Cauchy sequences in the spaces $C^k(I_T;H^{q+1-k}\cap H_0^1)$ and $C^k(I_T;H^{q+1-k}\cap H_*^1)$ ($k=0,\dots,q+1$), respectively, and their limits exist; denote them by $(\phi,\psi)$. By Theorem \ref{thm:strong_sol}, this limit is the unique strong solution. Finally, combining the stability estimate of Lemma \ref{qth_energy} with the induction hypothesis yields
	\begin{align*}
		&\sum_{k=0}^{q+1} \| \phi \|_{C^k\left(I_T;H^{q+1-k}\cap H_0^1\right)}
		+\sum_{k=0}^{q+1} \| \psi \|_{C^k\left(I_T;H^{q+1-k}\cap H_*^1\right)}
		\\
		&\lesssim
		\| \phi_0 \|_{H^{q+1}}+ \| \phi_1 \|_{H^{q}}+\| \psi_0 \|_{H^{q+1}}+ \| \psi_1 \|_{H^{q}}.
	\end{align*}
so the theorem holds for $q+1$. This completes the induction.
\end{proof}
\section{Fourier spectral Galerkin semi-discrete and fully discrete schemes}\label{s4}

In this section, we consider the numerical approximation of \eqref{Timoshenko}--\eqref{IC}. 
For the spatial discretization, a Fourier spectral Galerkin method is employed.
Let
\begin{align*}
	\phi^N(x,t)=\sum_{j=1}^N a_{j}(t)w_j(x),\quad \psi^N(x,t)=\sum_{j=1}^N b_{j}(t)v_j(x).   
\end{align*}
Assuming the regularity stated in Theorem~\ref{regularity}, we establish the following error bound.

\begin{theorem}\label{space_esti}
	Let \(q\in\mathbb{N}\) with \(q\geq 2\). Assume that the hypotheses of Theorem~\ref{regularity} hold. Let \((\phi,\psi)\) be the unique strong solution of \eqref{Timoshenko}--\eqref{IC}. Then
	\begin{align*}
		&\| \phi - \phi^N \|_{C(I_T;L^2)}
		+\| \psi - \psi^N \|_{C(I_T;L^2)} \\
		&\lesssim
		N^{-q}\left[
		\| \phi_0 \|_{H^q}+\| \phi_1 \|_{H^{q-1}}
		+\| \psi_0 \|_{H^q}+\| \psi_1 \|_{H^{q-1}}
		\right].
	\end{align*}
\end{theorem}

\begin{proof}
	Since $\phi^N$ and $\psi^N$ are the Fourier--Galerkin truncations of the strong solution $(\phi,\psi)$, their coefficients are given by
	$a_j(t)=(\phi(\cdot,t),w_j)$ and $b_j(t)=(\psi(\cdot,t),v_j)$. Using \eqref{eng_equ} together with the stability estimate for the energy, we obtain
	\begin{align*}
		\| \phi - \phi^N \|^2_2 &= \sum_{j=N+1}^{\infty} |a_j|_2^2 
		= \sum_{j=N+1}^{\infty} \frac{c^{2q}_j}{c^{2q}_j} |a_j|_2^2 
		\leq \frac{1}{c^{2q}_N} \sum_{j=N+1}^{\infty} c^{2q}_j |a_j|_2^2 \\
		&\lesssim N^{-2q}  [\| \phi_0 \|_{H^q}+ \| \phi_1 \|_{H^{q-1}}+\| \psi_0 \|_{H^q}+ \| \psi_1 \|_{H^{q-1}}]^2,   
	\end{align*}
	and similarly,
	\begin{align*}
		\| \psi - \psi^N \|^2_2 
		\lesssim N^{-2q} [ \| \phi_0 \|_{H^q}+ \| \phi_1 \|_{H^{q-1}}+\| \psi_0 \|_{H^q}+ \| \psi_1 \|_{H^{q-1}}]^2,   
	\end{align*}
	which yield the desired estimate.
\end{proof}

To discretize the coupled second-order integro-differential system \eqref{gkIC}-\eqref{ab_ode} in time, we employ the centered difference scheme for the second-order derivatives and the composite trapezoidal rule for the convolution integrals, thereby achieving second-order accuracy for both the differential and memory terms.

Let the time interval $I_T$ be uniformly partitioned with time step $\tau>0$, and denote the grid points by $t_n=n\tau$, where $n=0,\dots,M$ and $M\tau=T$.

Let $a_i^n$ and $b_i^n$ denote the numerical approximations to $a_i(t_n)$ and $b_i(t_n)$, respectively. 
Let $g_n:=g(t_n)$ denote the values of the kernel at the grid points. Using the centered difference approximation for the second derivatives, one has
    \begin{equation}\label{diff_approx}
        a_i^{\prime \prime}(t_n) \approx \frac{a_i^{n+1} - 2a_i^n + a_i^{n-1}}{\tau^2}, \ b_i^{\prime \prime}(t_n) \approx \frac{b_i^{n+1} - 2b_i^n + b_i^{n-1}}{\tau^2},\ n=1,\dots,M-1.
    \end{equation}
For \(i=1,\dots,N\), denote the convolution term by $\mathcal{I}[d_i](t):=\int_0^t g(t-s)d_i(s)\,\mathrm{d}s$. At time \(t_n\), we approximate this integral by the composite trapezoidal rule:
\begin{equation}\label{int_approx}
\mathcal{I}[d_i](t_n)\approx \mathcal{Q}_n[d_i]
:=\tau\Big(\frac12 g_n d_i^0+\sum_{j=1}^{n-1} g_{n-j}d_i^j+\frac12 g_0 d_i^n\Big).
\end{equation}
where \(d_i^j:=c_i a_i^j+b_i^j\), \(j=0,\dots,M\).

Substituting \eqref{diff_approx} and \eqref{int_approx} into \eqref{ab_ode}, and evaluating at $t=t_n$, 
one obtains the fully discrete scheme. For the first component $a_i$, 
by the linearity of $\mathcal{Q}_n$ we have
    \begin{align}\label{scheme_a}
        \frac{a_i^{n+1} - 2a_i^n + a_i^{n-1}}{\tau^2} & = -\kappa_1 c_i d^n_i + \kappa_1 c_i \mathcal{Q}_n[d_i] \nonumber \\
        &= -\kappa_1 c_i^2 a_i^n - \kappa_1 c_i b_i^n + \kappa_1 c_i^2 \mathcal{Q}_n[a_i] + \kappa_1 c_i \mathcal{Q}_n[b_i],
    \end{align}
i.e.,
    \begin{equation*}
        a_i^{n+1} = 2a_i^n - a_i^{n-1} + \tau^2 \left[ -\kappa_1 c_i \left( c_i a_i^n + b_i^n \right) + \kappa_1 c_i \left( c_i \mathcal{Q}_n[a_i] + \mathcal{Q}_n[b_i] \right) \right].
    \end{equation*}
Similarly, for the second component $b_i$, one gets
    \begin{equation}\label{scheme_b}
        \frac{b_i^{n+1} - 2b_i^n + b_i^{n-1}}{\tau^2} = - \bar{b} c_i^2 b_i^n -\kappa_2 d_i^n + \kappa_2 \mathcal{Q}_n[d_i],
    \end{equation}
i.e.,
    \begin{equation*}
        b_i^{n+1} = 2b_i^n - b_i^{n-1} + \tau^2 \left[ -\kappa_2 c_i a_i^n - (\bar{b} c_i^2 + \kappa_2) b_i^n + \kappa_2 \left( c_i \mathcal{Q}_n[a_i] + \mathcal{Q}_n[b_i] \right) \right].
    \end{equation*}
The schemes \eqref{scheme_a} and \eqref{scheme_b} define the time-stepping recursion for \(n\ge 1\).

Since the centered difference scheme is a three-level method, computing $a_i^1$ and $b_i^1$ requires the values at $n=0$ and $n=-1$, but the numerical values at $t_{-1}$ are not defined. We therefore set
    \begin{equation}\label{taylor_start}
        \begin{aligned}
            a_i^1 := a_i(0) + \tau a_i^{\prime}(0) + \frac{\tau^2}{2} a_i^{\prime \prime}(0),\quad
            b_i^1 := b_i(0) + \tau b_i^{\prime}(0) + \frac{\tau^2}{2} b_i^{\prime \prime}(0).
        \end{aligned}
    \end{equation}
Here $a_i(0)$, $b_i(0)$, $a_i^{\prime}(0)$, and $b_i^{\prime}(0)$ are directly determined by the initial conditions \eqref{gkIC}. The initial second derivatives $a_i^{\prime \prime}(0)$ and $b_i^{\prime \prime}(0)$ are obtained by evaluating the original system \eqref{ab_ode} at $t=0$:
    \begin{equation}\label{initial_acc}
        \left\{\begin{aligned}
           a_{i}^{\prime \prime}(0) &= -\kappa_1 c_i^2 a_{i}(0) - \kappa_1 c_i b_{i}(0), \\
           b_{i}^{\prime \prime}(0) &= -\kappa_2 c_i a_{i}(0) - \bar{b} c_i^ 2 b_{i}(0) - \kappa_2 b_{i}(0).
        \end{aligned}\right.
    \end{equation}
Substituting \eqref{initial_acc} into \eqref{taylor_start}, we can compute $a_i^1$ and $b_i^1$. After obtaining $a_j^n$ and $b_j^n$, the fully discrete approximation is defined by
\begin{align}\label{space_time_ns}
	\phi^{N,\tau}(x,t_n)=\sum_{j=1}^N a^n_{j} w_j(x),\quad \psi^{N,\tau}(x,t_n)=\sum_{j=1}^N b^n_{j} v_j(x),\ n=0,\dots,M.  
\end{align}

In the same spirit as the well-posedness and regularity analysis, the convergence analysis is based on a discrete energy method. We first derive the error equation and then estimate a suitably-defined discrete error energy.

Evaluating \eqref{ab_ode} at $t=t_n$ gives
    \begin{align}
      &  a_i''(t_n)=-\kappa_1 c_i\,d_i(t_n)+\kappa_1 c_i\,\mathcal I[d_i](t_n),\label{exact_a_n}\\
      &  b_i''(t_n)=-\bar b\,c_i^2 b_i(t_n)-\kappa_2d_i(t_n)+\kappa_2\mathcal I[d_i](t_n).\label{exact_b_n}
    \end{align}
Set $\mathcal T[f](t_n)
:=
\tau\left[\frac12 g_n f(t_0)+\sum_{j=1}^{n-1}g_{n-j} f(t_j) +\frac12 g_0 f(t_n)\right]$
for any suitable function~$f$. 
Set $r_{\text{int,} i}^n :=  \mathcal{T}[d_i](t_n) - \mathcal{I}[d_i](t_n)$ and
\begin{align*}  
r_{\text{diff,} a_i}^n :=& \frac{a_i(t_{n+1}) - 2a_i(t_n) + a_i(t_{n-1})}{\tau^2} - a_i''(t_n), \\
r_{\text{diff,} b_i}^n :=& \frac{b_i(t_{n+1}) - 2b_i(t_n) + b_i(t_{n-1})}{\tau^2} - b_i''(t_n).
\end{align*}  
Then the local truncation errors associated with $a_i$ and $b_i$ are
    \begin{align}\label{Ra_def}
	 R_{a,i}^n := r_{\text{diff,} a_i}^n 
	- \kappa_1 c_i r_{\text{int,} i}^n,\quad
	  R_{b,i}^n = r_{\text{diff,} b_i}^n 
	- \kappa_2 r_{\text{int,} i}^n.
\end{align}    
Substituting \eqref{exact_a_n} and \eqref{exact_b_n} into \eqref{Ra_def},  one gets
    \begin{align}
     &   \frac{a_i(t_{n+1})-2a_i(t_n)+a_i(t_{n-1})}{\tau^2} + \kappa_1 c_i\,d_i(t_n)
        =
        \kappa_1 c_i\,\mathcal T[d_i]+R_{a,i}^n,\label{exact_insert_a}\\
      &  \frac{b_i(t_{n+1})-2b_i(t_n)+b_i(t_{n-1})}{\tau^2} + \bar b\,c_i^2 b_i(t_n) + \kappa_2d_i(t_n) =
        \kappa_2\mathcal T[d_i]+R_{b,i}^n.\label{exact_insert_b}
    \end{align}
Subtract \eqref{scheme_a} and \eqref{scheme_b} from \eqref{exact_insert_a} and \eqref{exact_insert_b} respectively, then multiply the resulting equations by $\rho_1$ and $\rho_2$;
one arrives at the error system
    \begin{equation}\label{discrete_energy}
        \left\{\begin{aligned}
          &  \rho_1 \frac{e_{a,i}^{n+1} - 2e_{a,i}^n + e_{a,i}^{n-1}}{\tau^2} + \kappa c_i e_{d,i}^n = \kappa c_i (\mathcal{T}[d_i](t_n)-\mathcal{Q}_n[d_i]) + \rho_1 R_{a,i}^n, \\
           & \rho_2 \frac{e_{b,i}^{n+1} - 2e_{b,i}^n + e_{b,i}^{n-1}}{\tau^2} + b c_i^2 e_{b,i}^n + \kappa e_{d,i}^n = \kappa (\mathcal{T}[d_i](t_n)-\mathcal{Q}_n[d_i]) + \rho_2 R_{b,i}^n.
        \end{aligned}\right.
    \end{equation}
where
$e_{a,i}^n:=a_i(t_n)-a_{i}^n$,
$e_{b,i}^n:=b_i(t_n)-b_{i}^n$,
and
$e_{d,i}^n := c_i e_{a,i}^n + e_{b,i}^n$.

We first estimate the truncation errors $R^n_{a,i}$ and $R^n_{b,i}$ in \eqref{Ra_def}, 
each of which consists of two parts: the error of the centered second-order difference approximation 
and the error of the composite trapezoidal rule. The next lemma gives a convenient integral representation 
of the quadrature error for the composite trapezoidal rule.

\begin{lemma}\label{lem:trap_integral_remainder_detail} (see \cite[Section 5.1]{Atkinson89})
Let $n\in\mathbb N$ and $f\in C^2([0,t_n];\mathbb{R})$.
Define the quadrature error by
    \[
        E_n[f]:=\sum_{j=0}^{n-1}\frac{\tau}{2}\big(f(t_j)+f(t_{j+1})\big) -\int_0^{t_n}f(s) \, \mathrm{d}s.
    \]
Then there exists a kernel $P_{n,\tau}$ such that $E_n[f]=\int_0^{t_n}P_{n,\tau}(s)\,f''(s)\,\mathrm{d}s,$
and
$|P_{n,\tau}(s)|\le \frac{\tau^2}{8}$
for all $s\in[0,t_n]$.  
\end{lemma}

The next lemma provides the truncation error estimates needed in the convergence analysis. The key step is to write the differential and quadrature remainders in integral form and then combine these representations with the higher-order energy estimates established in Lemma~\ref{qth_energy}.

\begin{lemma}\label{R_esti}
Assume that
$(\phi_0,\phi_1,\psi_0,\psi_1)\in(H^4\cap H_0^1)\times (H^{3}\cap H_0^1)\times(H^4\cap H_*^1)\times (H^{3}\cap H_*^1)$, $g\in C^{3}(I_T)$ and that the fourth-order compatibility conditions are satisfied, namely, \eqref{compatibility_even}--\eqref{compatibility_even_1} with $r=2$.
Then the local truncation errors $R_{a,i}^n$ and $R_{b,i}^n$ satisfy
    \begin{equation*}
        \tau\sum_{n=1}^{M-1}\sum_{i=1}^N |R_{a,i}^n|^2 \lesssim \tau^4,\quad \tau\sum_{n=1}^{M-1}\sum_{i=1}^N |R_{b,i}^n|^2 \lesssim \tau^4.
    \end{equation*}
\end{lemma}

\begin{proof}
	By the definitions of \(R_{a,i}^n\) and \(R_{b,i}^n\), 
	\begin{align*}
	&	|R_{a,i}^n|^2
		=
		|r_{\text{diff,}a_i}^n-\kappa_1 c_i r_{\text{int,}i}^n|^2
		\le
		2|r_{\text{diff,}a_i}^n|^2+2\kappa_1^2 c_i^2 |r_{\text{int,}i}^n|^2,\\
	&	|R_{b,i}^n|^2
		=
		|r_{\text{diff,}b_i}^n-\kappa_2 r_{\text{int,}i}^n|^2
		\le
		2|r_{\text{diff,}b_i}^n|^2+2\kappa_2^2 |r_{\text{int,}i}^n|^2.
	\end{align*}
	The two terms on the right-hand sides will be estimated separately.

	We first derive an integral representation for \(r_{\text{diff,}a_i}^n\). Expanding
	\(a_i(t_{n+1})\) and \(a_i(t_{n-1})\) about \(t_n\) by Taylor's formula with an integral remainder,
	 then adding the resulting identities, one obtains
	\begin{equation*}
		r_{\text{diff,}a_i}^n
		=
		\frac{1}{6\tau^2}\left[
		\int_{t_n}^{t_n+\tau}(t_n+\tau-s)^3 a_i^{(4)}(s)\,\mathrm{d}s
		+
		\int_{t_n-\tau}^{t_n}(s-(t_n-\tau))^3 a_i^{(4)}(s)\,\mathrm{d}s
		\right].
	\end{equation*}
	Hence a Cauchy-Schwarz inequality yields
	\[
	|r_{\text{diff,}a_i}^n|^2
	\le
	\frac{\tau^3}{126}
	\int_{t_{n-1}}^{t_{n+1}}|a_i^{(4)}(s)|^2\,\mathrm{d}s.
	\]
	Note that $a_i^{(4)}(t)
	=
	\frac{\mathrm{d}^4}{\mathrm{d}t^4}(\phi^N(\cdot,t),w_i)
	=
	(\partial_t^4\phi^N(\cdot,t),w_i)$. From Lemma~\ref{qth_energy} one has
	\begin{equation}\label{est_diff_suma_appB}
		\begin{aligned}
			&\tau\sum_{n=1}^{M-1}\sum_{i=1}^N |r_{\text{diff,}a_i}^n|^2
			\le
			\frac{\tau^4}{126}
			\sum_{i=1}^N\sum_{n=1}^{M-1}
			\int_{t_{n-1}}^{t_{n+1}}|a_i^{(4)}(s)|^2\,\mathrm{d}s 
			\le
			C\tau^4
			\sum_{i=1}^N\int_0^T |a_i^{(4)}(s)|^2\,\mathrm{d}s \\
			&\le
			CT\tau^4 \max_{t\in[0,T]}\|\partial_t^4\phi^N(\cdot,t)\|_2^2 
			\le
			C\tau^4
			\left[
			\|\phi_0\|_{H^4}+\|\psi_0\|_{H^4}
			+\|\phi_1\|_{H^3}+\|\psi_1\|_{H^3}
			\right]^2.
		\end{aligned}
	\end{equation}
	By the same argument,
	\begin{equation}\label{est_diff_sumb_appB}
		\tau\sum_{n=1}^{M-1}\sum_{i=1}^N |r_{\text{diff,}b_i}^n|^2
		\le
		C\tau^4
		\left[
		\|\phi_0\|_{H^4}+\|\psi_0\|_{H^4}
		+\|\phi_1\|_{H^3}+\|\psi_1\|_{H^3}
		\right]^2.
	\end{equation}

Next, we estimate the quadrature truncation error.	
	In Lemma~\ref{lem:trap_integral_remainder_detail}, take $f(s)=g(t_n-s)d_i(s)$ for $s\in[0,t_n]$, then
	\[
	|r_{\text{int,}i}^n|
	\le
	\frac{\tau^2}{4}\|g\|_{C^2([0,T];\mathbb{R})}
	\int_0^T\left[|d_i(s)|+|d_i'(s)|+|d_i''(s)|\right]\,\mathrm{d}s.
	\]
	Consequently a Cauchy--Schwarz inequality gives
	\begin{equation}\label{sum_int_expansion_appB}
		\sum_{i=1}^N c_i^2 |r_{\text{int,}i}^n|^2
		\le
		C\tau^4
		\sum_{i=1}^N
		\int_0^T
		\left[
		|c_i d_i(s)|^2+|c_i d_i'(s)|^2+|c_i d_i''(s)|^2
		\right]\mathrm{d}s.
	\end{equation}
	Observe that $\partial_x\eta^N(x,s)=-\sum_{j=1}^N c_j d_j(s)w_j(x)$, 
	then by Lemma~\ref{qth_energy} one has
	\[
	\|\partial_x\eta^N(\cdot,s)\|_2^2
	=
	\sum_{i=1}^N |c_i d_i(s)|^2
	\lesssim
	\mathbf E_2^N.
	\]
	One can show similarly that $\sum_{i=1}^N |c_i d_i'(s)|^2 \lesssim \mathbf E_3^N,\ \sum_{i=1}^N |c_i d_i''(s)|^2 \lesssim \mathbf E_4^N$. Substituting these estimates into \eqref{sum_int_expansion_appB}, we obtain
	\begin{equation*}
		\sum_{i=1}^N c_i^2 |r_{\text{int,}i}^n|^2
		\le
		C\tau^4 \mathbf E_4^N
		\le
		C\tau^4
		\left(
		\|\phi_0\|_{H^4}+\|\psi_0\|_{H^4}
		+\|\phi_1\|_{H^3}+\|\psi_1\|_{H^3}
		\right)^2
		\le
		C\tau^4,
	\end{equation*}
which together with \eqref{est_diff_suma_appB} and \eqref{est_diff_sumb_appB} 
yields the bounds for \(R_{a,i}^n\) and \(R_{b,i}^n\).
\end{proof}

We next analyze the convergence of the fully discrete scheme.

\begin{theorem}\label{thm:full_convergence}
Let \(q\in\mathbb{N}\) with \(q\geq 4\). Assume that the hypotheses of Theorem~\ref{regularity} hold and the mesh sizes satisfy $\tau=\mathcal O(N^{-1})$. Then the Timoshenko system \eqref{Timoshenko}--\eqref{IC} admits a unique strong solution \((\phi,\psi)\), and the fully discrete solution \eqref{space_time_ns} satisfies
    \begin{equation*}
        \begin{aligned}
            \| \phi - \phi^{N,\tau} \|_{L^\infty_\tau(L^2)}:=
            \max_{0\le n\le M}\left\|\phi(\cdot,t_n)-\phi^{N,\tau}(\cdot,t_n)\right\|_{2} &\lesssim \left( N^{-q} + \tau^2 \right), \\
            \| \psi - \psi^{N,\tau} \|_{L^\infty_\tau(L^2)}:=
            \max_{0\le n\le M}\left\|\psi(\cdot,t_n)-\psi^{N,\tau}(\cdot,t_n)\right\|_{2} &\lesssim \left( N^{-q} + \tau^2 \right).
        \end{aligned}
    \end{equation*}
\end{theorem}
\begin{proof}
We prove the error estimate only for $\phi$ since the argument for $\psi$ is analogous. 
Split the total error into the spatial projection error $E_{I} := \| \phi(\cdot, t_n) - \phi^N(\cdot, t_n) \|_{2} $ and the temporal discretization error $E_{II}:=\left \|\phi^N(\cdot,t_n)-\phi^{N,\tau}(\cdot,t_n) \right \|_{2}$. Then
  $$
\left\|\phi(\cdot,t_n)-\phi^{N,\tau}(\cdot,t_n)\right\|_{2} \le E_{I} + E_{II}.
$$

By Theorem \ref{space_esti}, for every $t_n$ we have
    \begin{equation}\label{est_I}
        E_{I} = \| \phi(\cdot, t_n) - \phi^N(\cdot, t_n) \|_{2} \le C N^{-q}.
    \end{equation} 
Thus, we need only estimate $E_{II}$. Define the discrete energy at time level $n$ by
    \begin{equation*}
        E_i^{n} := \frac{1}{2}\left[\rho_1 |\delta_\tau e_{a,i}^{n+1}|^2 + \rho_2 |\delta_\tau e_{b,i}^{n+1}|^2 
        +  b c_i^2 e_{b,i}^{n+1} e_{b,i}^n 
        +  \kappa e_{d,i}^{n+1} e_{d,i}^n\right],
    \end{equation*}
where $\delta_\tau e_{a,i}^{n+1} :=  \frac{e_{a,i}^{n+1} - e_{a,i}^{n}}{\tau}$ and  $\delta_\tau e_{b,i}^{n+1} :=  \frac{e_{b,i}^{n+1} - e_{b,i}^{n}}{\tau}$.     
Multiplying the first equation in \eqref{discrete_energy} by
$\delta_{2\tau} e_{a,i}^n := \frac{e_{a,i}^{n+1} - e_{a,i}^{n-1}}{2\tau}$,
the second equation by
$\delta_{2\tau} e_{b,i}^n := \frac{e_{b,i}^{n+1} - e_{b,i}^{n-1}}{2\tau}$,
and then adding the two resulting identities, we obtain
\begin{equation}\label{energy_id}
        \frac{E_i^n - E_i^{n-1}}{\tau} = \kappa (\mathcal{T}[d_i](t_n)-\mathcal{Q}_n[d_i]) \delta_{2\tau} e_{d,i}^n + \rho_1 R_{a,i}^n \delta_{2\tau} e_{a,i}^n + \rho_2 R_{b,i}^n \delta_{2\tau} e_{b,i}^n.
\end{equation}

To guarantee the positivity of the discrete energy, we expand $E_i^n$ algebraically. Using the identity
$uv = \frac{1}{2}(u^2+v^2) - \frac{1}{2}(u-v)^2$,
we rewrite the cross terms as follows:
\[
        E_i^n = E_{1,i}^n 
       + \frac{1}{4} b c_i^2 \left( |e_{b,i}^{n+1}|^2 + |e_{b,i}^n|^2 \right) 
        + \frac{1}{4} \kappa \left( |e_{d,i}^{n+1}|^2 + |e_{d,i}^n|^2 \right),
    \]
where 
\begin{align*}
        E_{1,i}^n 
        := & \frac{1}{2} \left[ \rho_1 |\delta_\tau e_{a,i}^{n+1}|^2 + \rho_2 |\delta_\tau e_{b,i}^{n+1}|^2 - \frac{\tau^2}{2} b c_i^2 |\delta_\tau e_{b,i}^{n+1}|^2 - \frac{\tau^2}{2} \kappa |c_i \delta_\tau e_{a,i}^{n+1} + \delta_\tau e_{b,i}^{n+1}|^2 \right]\\
        \ge& \frac{1}{2} \left( \rho_1 - \tau^2 \kappa c_i^2 \right) |\delta_\tau e_{a,i}^{n+1}|^2 
        + \frac{1}{2} \left( \rho_2 - \frac{\tau^2}{2} b c_i^2-\tau^2 \kappa \right) |\delta_\tau e_{b,i}^{n+1}|^2.
\end{align*}
Thus, under the mesh condition \(\tau=O(N^{-1})\), there exists a constant \(C>0\) such that
\begin{equation}\label{norm_equiv}
        E_i^n \ge C \left( |\delta_\tau e_{a,i}^{n+1}|^2 + |\delta_\tau e_{b,i}^{n+1}|^2 + |e_{d,i}^n|^2 + |e_{d,i}^{n+1}|^2 + c_i^2 |e_{b,i}^n|^2 \right).
\end{equation}
In particular, the combination \( |e_{d,i}^n|^2 + c_i^2 |e_{b,i}^n|^2 \) controls \( |e_{a,i}^n|^2 \).

Multiply both sides of \eqref{energy_id} by \(\tau\) and sum over \(n=1, \dots, m\), where \(1\le m\le M-1\), 
to get
    \begin{equation}\label{eq:sum_energy_id_short}
        E_i^m-E_i^0
        =
        \kappa \mathfrak I_i^m
        +\rho_1\tau\sum_{n=1}^{m} R_{a,i}^n\delta_{2\tau}e_{a,i}^n
        +\rho_2\tau\sum_{n=1}^{m} R_{b,i}^n\delta_{2\tau}e_{b,i}^n,
    \end{equation}
where $A_n:=\mathcal T[d_i](t_n)-\mathcal Q_n[d_i]$ and $\mathfrak I_i^m:= \frac12\sum_{n=1}^{m}A_n\bigl(e_{d,i}^{n+1}-e_{d,i}^{n-1}\bigr)= \mathfrak I_{\mathrm{end}}^m+\mathfrak I_{\mathrm{mid}}^m$; 
here $\mathfrak I_{\mathrm{end}}^m = \frac12\Big(A_{m}e_{d,i}^{m+1}+A_{m-1}e_{d,i}^{m}-A_{1}e_{d,i}^{0}-A_{2}e_{d,i}^{1}\Big)$
and
$\mathfrak I_{\mathrm{mid}}^m = \frac12\sum_{r=2}^{m-1}(A_{r-1}-A_{r+1})e_{d,i}^{r}$
are the endpoint and interior contributions, respectively.

The nonnegativity of the trapezoidal weights and a Cauchy-Schwarz inequality yield
    \begin{equation}\label{eq:An_bound_short}
        |A_n|^2
        \le
        C\tau\sum_{j=0}^{n} g_{n-j}|e_{d,i}^j|^2,
        \qquad n\ge 1,
    \end{equation}
where the constant \(C>0\) depends only on \(T\) and \(g_0\). Applying Young's inequality with parameter \(\varepsilon>0\) to each term in \(\mathfrak I_{\mathrm{end}}^m\) and then using \eqref{eq:An_bound_short}, 
we obtain
    \begin{equation}\label{eq:Iend_bound_short}
        |\kappa \mathfrak I_{\mathrm{end}}^m|
        \le
        \frac{C}{\varepsilon}\tau\sum_{n=0}^{m-1}E_i^n
        +\varepsilon C\bigl(E_i^m+E_i^0\bigr).
    \end{equation}
Here we have used the fact that the discrete energy $E_i^n$ controls $|e_{d,i}^n|^2$.

Again using the identity
$uv = \frac{1}{2}(u^2+v^2) - \frac{1}{2}(u-v)^2$,
together with the nonnegativity and monotonicity of the kernel \(g\), we decompose \(\mathfrak I_{\mathrm{mid}}^m\) into several terms. The terms involving squared differences are nonpositive and may therefore be discarded. As a consequence, one gets
    \begin{equation*}
        \mathfrak I_{\mathrm{mid}}^m
        \le
        C\tau\sum_{n=0}^{m}|e_{d,i}^n|^2.
    \end{equation*}
Invoking once again the control of $|e_{d,i}^n|^2$ by the discrete energy, we  deduce that
    \begin{equation}\label{eq:Imid_energy_bound_short}
        |\kappa \mathfrak I_{\mathrm{mid}}^m|
        \le
        C\tau\sum_{n=0}^{m-1}E_i^n.
    \end{equation}

Combining \eqref{eq:Iend_bound_short} and \eqref{eq:Imid_energy_bound_short} gives
    \begin{equation}\label{eq:I_bound_short}
        |\kappa \mathfrak I_i^m|
        \le
        \frac{C}{\varepsilon}\tau\sum_{n=0}^{m-1}E_i^n
        +\varepsilon C\bigl(E_i^m+E_i^0\bigr).
    \end{equation}
By choosing \(\varepsilon>0\) sufficiently small, one can absorb the right-hand side  \(E_i^m\) term 
into the energy term on the left-hand side.

We now estimate the truncation-error terms in \eqref{eq:sum_energy_id_short}. By Young's inequality,
    \begin{equation}\label{Ra_term_bound}
        \begin{aligned}
            \rho_1\tau\sum_{n=1}^m R_{a,i}^n\delta_{2\tau}e_{a,i}^n
            &\le
            \frac{\rho_1\tau}{2}\sum_{n=1}^m |R_{a,i}^n|^2
            +\frac{\rho_1\tau}{2}\sum_{n=1}^m|\delta_{2\tau}e_{a,i}^n|^2 \\
            &\le
            \frac{\rho_1\tau}{2}\sum_{n=1}^m |R_{a,i}^n|^2
            +\frac{\rho_1\tau}{4}\sum_{n=1}^m\left(|\delta_\tau e_{a,i}^{n+1}|^2+|\delta_\tau e_{a,i}^{n}|^2\right).
        \end{aligned}
    \end{equation}
Similarly,
    \begin{equation}\label{Rb_term_bound}
        \begin{aligned}
            \rho_2\tau\sum_{n=1}^m R_{b,i}^n\delta_{2\tau}e_{b,i}^n
            \le
            \frac{\rho_2\tau}{2}\sum_{n=1}^m |R_{b,i}^n|^2
            +\frac{\rho_2\tau}{4}\sum_{n=1}^m\left(|\delta_\tau e_{b,i}^{n+1}|^2+|\delta_\tau e_{b,i}^{n}|^2\right).
        \end{aligned}
    \end{equation}
Next, we estimate the second terms on the right-hand sides of \eqref{Ra_term_bound} and \eqref{Rb_term_bound}. By \eqref{norm_equiv}, there exists a constant $C>0$ such that, for all $n \in \mathbb{N}$,
    \begin{equation}\label{velocity_sum_control}
        \begin{aligned}
            \frac{\rho_1\tau}{4}\sum_{n=1}^m\left(|\delta_\tau e_{a,i}^{n+1}|^2+|\delta_\tau e_{a,i}^{n}|^2\right)
            +
            \frac{\rho_2\tau}{4}\sum_{n=1}^m\left(|\delta_\tau e_{b,i}^{n+1}|^2+|\delta_\tau e_{b,i}^{n}|^2\right)
            \le C \tau\sum_{n=0}^{m}E_i^n.
        \end{aligned}
    \end{equation}
Combining \eqref{eq:I_bound_short}, \eqref{velocity_sum_control}, and \eqref{eq:sum_energy_id_short}, and then choosing $\varepsilon$ and $\tau$ sufficiently small so that
$1-C\varepsilon-C\tau> \frac{1}{2}$,
we arrive at
    \begin{equation}\label{Ei_recursion_sum}
        E_i^m
        \lesssim
        E_i^0
        +
        \tau\sum_{n=0}^{m-1}E_i^n
        +
        \tau\sum_{n=1}^{m}\left(|R_{a,i}^n|^2+|R_{b,i}^n|^2\right).
\end{equation} 

Define the total energy by
$\mathbb{E}^n:=\sum_{i=1}^N E_i^n$, and sum \eqref{Ei_recursion_sum} over $i=1,\dots,N$.
One then has 
    \begin{equation*}
        \mathbb{E}^m
        \lesssim
        \mathbb{E}^0
        +
        \tau\sum_{n=0}^{m-1}\mathbb{E}^n
        +
        \tau\sum_{n=1}^{m}\sum_{i=1}^N\left(|R_{a,i}^n|^2+|R_{b,i}^n|^2\right).
    \end{equation*}
The truncation-error estimate of Lemma \ref{R_esti} gives
    \begin{equation}\label{gronwall_ready2}
        \mathbb{E}^m
        \lesssim
        \mathbb{E}^0
        +
        \tau\sum_{n=0}^{m-1}\mathbb{E}^n
        +
        \tau^4.
    \end{equation}
Applying a discrete Gronwall's inequality \cite[Corollary 2.1.19]{Bru04} to \eqref{gronwall_ready2}, we obtain
    \begin{equation}\label{E_bound_final}
        \max_{0\le m\le M-1}\mathbb{E}^m
        \lesssim
        \left(\mathbb{E}^0+\tau^4\right).
    \end{equation}

We next estimate $\mathbb{E}^0$. Since the numerical initial values are taken as the exact initial values,
$a_{i}^0=a_i(0)$ and $b_{i}^0=b_i(0)$,
it follows that
$e_{a,i}^0 = e_{b,i}^0 = e_{d,i}^0 = 0$.
Then
    \begin{equation}\label{Ei0_simplify}
        E_i^0=\frac12\rho_1|\delta_\tau e_{a,i}^1|^2+\frac12\rho_2|\delta_\tau e_{b,i}^1|^2.
    \end{equation}
Expanding $a_i$ and $b_i$ at $t=\tau$ by Taylor's formula with integral remainder, and using the definitions of the first-step approximations $a_{i}^1$ and $b_{i}^1$, we obtain
    \begin{equation}\label{e1_taylor}
        e_{a,i}^1=a_i(\tau)-a_{i}^1=\frac{1}{2} \int_{0}^{\tau}(\tau-s)^2 a_i^{(3)}(s)\,\mathrm{d}s,
        \
        e_{b,i}^1=b_i(\tau)-b_{i}^1=\frac{1}{2} \int_{0}^{\tau}(\tau-s)^2 b_i^{(3)}(s)\,\mathrm{d}s.
    \end{equation}
Since $e_{a,i}^0=e_{b,i}^0=0$, one has
$
\delta_\tau e_{a,i}^1=\frac{e_{a,i}^1-e_{a,i}^0}{\tau}=\frac{e_{a,i}^1}{\tau},
\
\delta_\tau e_{b,i}^1=\frac{e_{b,i}^1-e_{b,i}^0}{\tau}=\frac{e_{b,i}^1}{\tau}.
$
Thus, using \eqref{e1_taylor}, we obtain
    \begin{equation}\label{delta_e1_taylor}
        | \delta_\tau e_{a,i}^1 |^2 \le \frac{\tau^3}{180} \int_0^{\tau} | a_i^{(3)}(s) |^2 \,\mathrm{d}s,
        \qquad
        | \delta_\tau e_{b,i}^1 |^2 \le \frac{\tau^3}{180} \int_0^{\tau} | b_i^{(3)}(s) |^2 \,\mathrm{d}s.
    \end{equation}
Now sum \eqref{Ei0_simplify} over $i=1,\dots,N$ and use \eqref{delta_e1_taylor} to get
    \begin{equation*}
        \mathbb E^0
        =
        \sum_{i=1}^N E_i^0
        \lesssim
        \tau ^3 \int_{0}^{\tau} \sum_{i=1}^N (|a_i^{(3)}(s) |^2  + |b_i^{(3)}(s) |^2) \,\mathrm{d}s,
    \end{equation*}
But
$a_i^{(3)}(t)=(\partial_t^3\phi^N(\cdot,t),w_i)$
and
$b_i^{(3)}(t)=(\partial_t^3\psi^N(\cdot,t),v_i)$,
so invoking Lemma \ref{qth_energy} we obtain
\[
\int_{0}^{\tau} \sum_{i=1}^N (|a_i^{(3)}(s) |^2  + |b_i^{(3)}(s) |^2) \,\mathrm{d}s \le \tau \mathbf{E}_3^N  \le C \tau .
\]
Hence
$\mathbb{E}^0\le C\tau^4$.
Substituting this bound into \eqref{E_bound_final} gives 
$$
\max_{0\le m\le M-1}\mathbb{E}^m\le C\tau^4.
$$

Finally, by \eqref{norm_equiv}, one has $|e_{a,i}^n|^2\le C E_i^{n-1}$. Therefore,
\[
E^2_{II}=\|\phi^N(\cdot,t_n)-\phi^{N,\tau}(\cdot,t_n)\|^2_{2} = \sum_{i=1}^N |e_{a,i}^n|^2 \le C \sum_{i=1}^N E_i^{n-1}=C\mathbb{E}^{n-1}\le C\tau^4,
\]
which together with the spatial error estimate \eqref{est_I}, yields
\[
\|\phi(\cdot,t_n)-\phi^{N,\tau}(\cdot,t_n)\|_{2}
\le C\left(N^{-q}+\tau^2\right),
\quad 0\le n\le M.
\]
The estimate for the $\psi$-component can be proved in the same way.
\end{proof}
\section{Numerical experiments}\label{s5}
In this section, we give two numerical experiments. The first one verifies the convergence rates predicted by Theorem~\ref{thm:full_convergence}, while the second compares the nonlocal and local models and highlights the energy dissipation and memory effects induced by the Volterra term. All numerical experiments were carried out on a desktop computer equipped with an Intel Core i9-10900K CPU and implemented in Baltamatica (version 2.0).
\begin{example}\label{ex1}
We take
$
T=l=\rho_1=\rho_2=\kappa=b=1$ and 
$g(t)=\exp\!\left(-\frac{129}{128}t\right)
$
(the choice of the kernel function \(g\) is motivated by \cite[Example 3.7]{alves2019modeling}). The initial conditions in \eqref{IC} are chosen as
\[
\phi_0(x)=\phi_1(x)=x^3(1-x)^3,\ \psi_0(x)=x^4(1-x)^4-\frac{1}{630},\
\psi_1(x)=x^2(1-x)^2-\frac{1}{30},
\]
for $x\in [0,1]$. It is easy to verify that $(\phi_0,\phi_1,\psi_0,\psi_1)\in
(H^4\cap H_0^1)\times(H^3\cap H_0^1)\times(H^4\cap H_*^1)\times(H^3\cap H_*^1)$ and that these data satisfy the fourth-order compatibility conditions \eqref{compatibility_even}--\eqref{compatibility_even_1} with \(r=2\), but do not satisfy the fifth-order compatibility conditions.
\end{example}

Since an explicit exact solution is not available for this problem, we use a sufficiently fine numerical solution as the reference solution. More precisely, when testing the temporal convergence rate, we take
$N_{\mathrm{ref}}=2^5$ and $M_{\mathrm{ref}}=2^{13}$, with the corresponding time step denoted by $\tau_{\mathrm{ref}}$.
When testing the spatial convergence rate, we take
$N_{\mathrm{ref}}=2^9$ and $M_{\mathrm{ref}}=2^{13}$.
The discrete errors and experimental convergence rates are then computed accordingly.

\medskip
\noindent
\textbf{(1) Test of temporal convergence rates}

For the temporal meshes
$M=2^6,\,2^7,\,2^8,\,2^9,\,2^{10}$,
we compute the fully discrete solutions
\((\phi^{N_{\mathrm{ref}},\tau},\psi^{N_{\mathrm{ref}},\tau})\),
and compare them with \((\phi^{N_{\mathrm{ref}},\tau_{\mathrm{ref}}},\psi^{N_{\mathrm{ref}},\tau_{\mathrm{ref}}})\)
at the common time nodes. We define the temporal errors by
\begin{align*}
&\mathrm{Err}_\phi^\tau(M):=
\max_{0\le n\le M}
\left\|
\phi^{N_{\mathrm{ref}},\tau_{\mathrm{ref}}}(\cdot,t_n)
-
\phi^{N_{\mathrm{ref}},\tau}(\cdot,t_n)
\right\|_{2},\\
&\mathrm{Err}_\psi^\tau(M):=
\max_{0\le n\le M}
\left\|
\psi^{N_{\mathrm{ref}},\tau_{\mathrm{ref}}}(\cdot,t_n)
-
\psi^{N_{\mathrm{ref}},\tau}(\cdot,t_n)
\right\|_{2}.
\end{align*}
The convergence orders between two consecutive temporal meshes is estimated by
$
\mathrm{Order}_\tau^\phi:=
\frac{\ln\bigl(\mathrm{Err}_\phi^\tau(M_1)/\mathrm{Err}_\phi^\tau(M_2)\bigr)}
{\ln(M_2/M_1)}$ and
$\mathrm{Order}_\tau^\psi:=
\frac{\ln\bigl(\mathrm{Err}_\psi^\tau(M_1)/\mathrm{Err}_\psi^\tau(M_2)\bigr)}
{\ln(M_2/M_1)}$,
where \(M_1\) and \(M_2\) denote two consecutive temporal mesh parameters. The corresponding numerical results are listed in Table \ref{tab:time_convergence_refined}. It can be observed that the centered difference-composite trapezoidal scheme constructed in this paper achieves second-order convergence in time both for \(\phi\) and \(\psi\).

\begin{table}[H]
	\centering
	\caption{Temporal errors and convergence orders}
	\label{tab:time_convergence_refined}
	\begin{tabular}{ccccc}
		\toprule
		\multirow{2}{*}{$M$} & \multicolumn{2}{c}{$\phi$} & \multicolumn{2}{c}{$\psi$} \\
		\cmidrule(lr){2-3}\cmidrule(lr){4-5}
		& $\mathrm{Err}_\phi^\tau(M)$ & $\mathrm{Order}_\tau^\phi$
		& $\mathrm{Err}_\psi^\tau(M)$ & $\mathrm{Order}_\tau^\psi$ \\
		\midrule
		$2^6$    & $1.15\mathrm{E}{-05}$ & --   & $1.20\mathrm{E}{-05}$ & --   \\
		$2^7$    & $2.86\mathrm{E}{-06}$ & $2.00$ & $2.99\mathrm{E}{-06}$ & $2.01$ \\
		$2^8$    & $7.15\mathrm{E}{-07}$ & $2.00$ & $7.47\mathrm{E}{-07}$ & $2.00$ \\
		$2^9$    & $1.78\mathrm{E}{-07}$ & $2.01$ & $1.86\mathrm{E}{-07}$ & $2.00$ \\
		$2^{10}$ & $4.40\mathrm{E}{-08}$ & $2.02$ & $4.60\mathrm{E}{-08}$ & $2.02$ \\
		\bottomrule
	\end{tabular}
\end{table}

\medskip
\noindent
\textbf{(2) Test of spatial convergence rates}

For the numbers of spatial basis functions
$N=2,\,4,\,8,\,16,\,32$,
we compute the fully discrete solutions
\((\phi^{N,\tau_{\mathrm{ref}}},\psi^{N,\tau_{\mathrm{ref}}})\),
and compare them with the reference solution
\((\phi^{N_{\mathrm{ref}},\tau_{\mathrm{ref}}},\psi^{N_{\mathrm{ref}},\tau_{\mathrm{ref}}})\).
We define the spatial errors by
\begin{align*}
&\mathrm{Err}_\phi^N(N):=
\max_{0\le n\le M_{\mathrm{ref}}}
\left\|
\phi^{N_{\mathrm{ref}},\tau_{\mathrm{ref}}}(\cdot,t_n)
-
\phi^{N,\tau_{\mathrm{ref}}}(\cdot,t_n)
\right\|_{2},\\
&\mathrm{Err}_\psi^N(N):=
\max_{0\le n\le M_{\mathrm{ref}}}
\left\|
\psi^{N_{\mathrm{ref}},\tau_{\mathrm{ref}}}(\cdot,t_n)
-
\psi^{N,\tau_{\mathrm{ref}}}(\cdot,t_n)
\right\|_{2}.
\end{align*}
The convergence orders between two consecutive spatial truncation parameters are defined by
$
\mathrm{Order}_N^\phi:=
\frac{\ln\bigl(\mathrm{Err}_\phi^N(N_1)/\mathrm{Err}_\phi^N(N_2)\bigr)}
{\ln(N_2/N_1)}$ and
$
\mathrm{Order}_N^\psi:=
\frac{\ln\bigl(\mathrm{Err}_\psi^N(N_1)/\mathrm{Err}_\psi^N(N_2)\bigr)}
{\ln(N_2/N_1)}$,
where \(N_1\) and \(N_2\) denote two consecutive spatial truncation parameters. The corresponding numerical results are reported in Table \ref{tab:space_convergence_refined}. It can be seen that the errors for both \(\phi\) and \(\psi\) exhibit at least fourth-order convergence in space, which indicates the Fourier-Galerkin spatial discretization provides a highly accurate approximation.

\begin{table}[H]
\centering
\caption{Spatial errors and convergence orders}
\label{tab:space_convergence_refined}
\begin{tabular}{ccccc}
\toprule
\multirow{2}{*}{$N$} & \multicolumn{2}{c}{$\phi$} & \multicolumn{2}{c}{$\psi$} \\
\cmidrule(lr){2-3}\cmidrule(lr){4-5}
& $\mathrm{Err}_\phi^N(N)$ & $\mathrm{Order}_N^\phi$
& $\mathrm{Err}_\psi^N(N)$ & $\mathrm{Order}_N^\psi$ \\
\midrule
$2$  & $2.45\mathrm{E}{-03}$ & --   & $8.26\mathrm{E}{-04}$ & --   \\
$4$  & $2.09\mathrm{E}{-04}$ & $3.56$ & $7.38\mathrm{E}{-05}$ & $3.48$ \\
$8$  & $1.20\mathrm{E}{-05}$ & $4.12$ & $4.45\mathrm{E}{-06}$ & $4.05$ \\
$16$ & $5.80\mathrm{E}{-07}$ & $4.38$ & $2.22\mathrm{E}{-07}$ & $4.33$ \\
$32$ & $2.62\mathrm{E}{-08}$ & $4.47$ & $1.02\mathrm{E}{-08}$ & $4.44$ \\
\bottomrule
\end{tabular}
\end{table}

\begin{example}
In order to investigate the influence of the memory term on the free vibration response of a Timoshenko beam, and to compare the dynamical behavior of the nonlocal model \eqref{Timoshenko} with that of the local model obtained by setting $g\equiv 0$ in \eqref{Timoshenko}, we take
$
T=20, l=\rho_1=\rho_2=\kappa=b=1,
$
and use the same memory kernel as in Example \ref{ex1}. In addition, both models are assigned the same initial data:
\[
\phi_0(x)=\sin(\pi x),\quad \phi_1(x)=0,\quad
\psi_0(x)=0,\quad \psi_1(x)=0,\quad x\in[0,1].
\]
\end{example}

In the numerical experiment, we take $N=2^6$ and $M=2^{12}$. We compute the fully discrete solutions for both the local and nonlocal models, and compare the time evolution of the midpoint displacement $\phi(1/2,t)$, the midpoint rotation $\psi(1/2,t)$, and the midpoint shear variable
$\eta(1/2,t)$. In addition, in order to further examine the dissipative behavior of the system, we compute the energy of the fully discrete scheme:
    \begin{align*}
            E^{N,\tau}(t_n)
            :=
            \frac12
             \sum_{j=1}^N [\rho_1 |\delta_{2\tau}  a_j^n|^2
            +\rho_2 |\delta_{2\tau} b_j^n|^2
            +b c_j^2 |b_j^n|^2
            +\kappa G(t_n) |d_j^n|^2]
            +\frac{\kappa}{2}\,\mathcal M_n,
    \end{align*}
where 
$
        \mathcal M_n
        :=
        \tau\Big[
        \frac12 g_n \sum_{j=1}^N |d_j^n-d_j^0|^2
        +
        \sum_{m=1}^{n-1} g_{n-m}\sum_{j=1}^N |d_j^n-d_j^m|^2
        \Big]$, $n=1,\dots,M-1.$
        
\begin{figure}[htpb]
\centering

\begin{subfigure}[b]{0.47\textwidth}
    \centering
    \includegraphics[width=\textwidth]{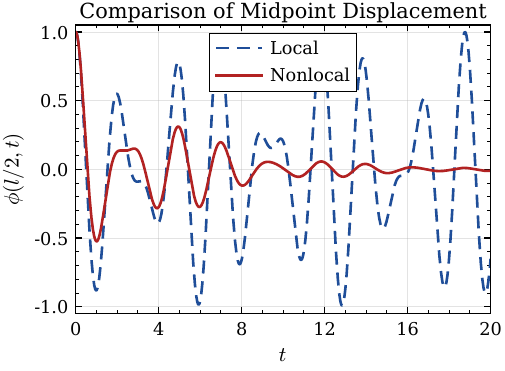}
    \caption{Midpoint displacement $\phi(1/2,t)$}
    \label{fig:mid_disp_compare}
\end{subfigure}
\hfill
\begin{subfigure}[b]{0.47\textwidth}
    \centering
    \includegraphics[width=\textwidth]{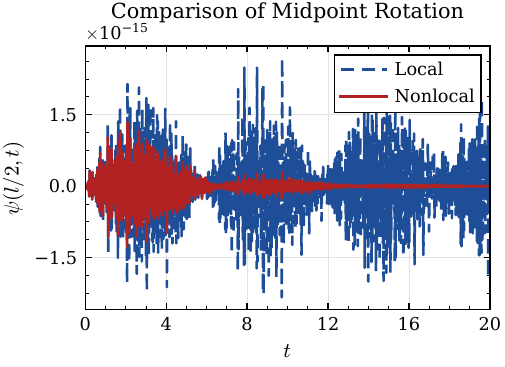}
    \caption{Midpoint rotation $\psi(1/2,t)$}
    \label{fig:mid_rot_compare}
\end{subfigure}

\vspace{0.4cm}

\begin{subfigure}[b]{0.47\textwidth}
    \centering
    \includegraphics[width=\textwidth]{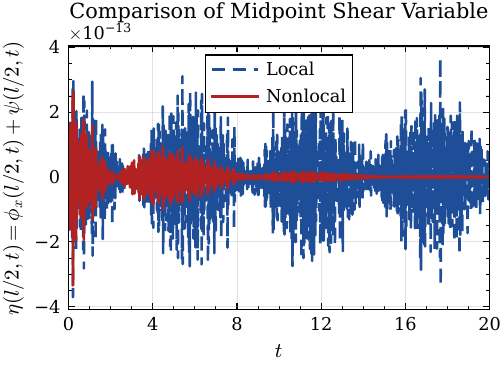}
    \caption{Midpoint shear variable $\eta(1/2,t)$}
    \label{fig:mid_shear_compare}
\end{subfigure}
\hfill
\begin{subfigure}[b]{0.47\textwidth}
    \centering
    \includegraphics[width=\textwidth]{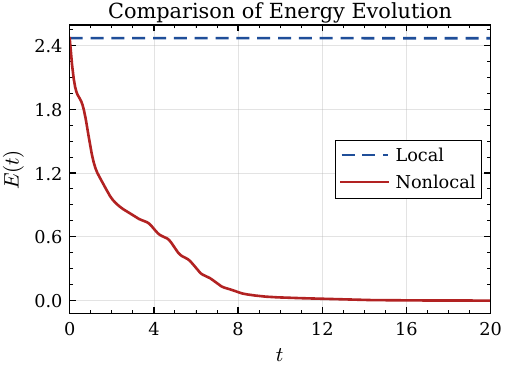}
    \caption{Energy curves}
    \label{fig:energy_compare}
\end{subfigure}

\caption{Comparison of midpoint dynamics and energy evolution.}
\label{fig:midpoint_results}
\end{figure}

The numerical results are shown in Figures \ref{fig:mid_disp_compare}--\ref{fig:energy_compare}. Figures \ref{fig:mid_disp_compare}, \ref{fig:mid_rot_compare} and \ref{fig:mid_shear_compare} show that, in the local model, the midpoint displacement, midpoint rotation and midpoint shear variable continue to oscillate over a certain range, exhibiting the typical behavior of free vibration. In contrast, in the nonlocal model, all three quantities gradually decay in time and eventually approach zero. This indicates that the memory term effectively captures the history dependence and dissipation mechanism of the material, thereby causing the free vibration of the beam to decay over time.

Furthermore, Figure \ref{fig:energy_compare} shows that the energy of the local model remains constant throughout the computation, reflecting its conservative structure, whereas the energy of the nonlocal model decreases monotonically and gradually tends to zero. This is consistent with the theoretical results: when $g\equiv 0$, the system reduces to the classical local Timoshenko model with conserved energy; when $g\not\equiv 0$, the convolution memory term introduces dissipation, causing the total energy of the system to decay over time.

In summary, the local model can only describe the undamped free vibration of an ideal elastic beam, whereas the nonlocal Timoshenko model considered in this paper, owing to the presence of the memory term, is capable of naturally characterizing energy dissipation and vibration decay in viscoelastic materials. It is therefore more physically realistic for describing the dynamical response of beam structures with history-dependent effects.
\section{Future work}\label{s6}

Several directions remain open for future investigation:

\begin{itemize}
    \item[(i)] \textbf{Dissipative structure at the discrete level.}  
    The fully discrete scheme analyzed in this paper is convergent, but a rigorous discrete dissipation theory has not yet been established. It would be interesting either to construct a discrete energy for the present scheme that is consistent with the continuous energy and prove its decay, or to design new structure-preserving schemes with built-in energy dissipation.

    \item[(ii)] \textbf{Higher-order time discretization.}  
    The present method is second-order accurate in time. A natural extension is to develop higher-order time-stepping schemes for memory terms and to establish the corresponding stability and convergence theory.

    \item[(iii)] \textbf{More general memory kernels and long-time dynamics.}  
    It would also be of interest to consider more general memory kernels, such as weakly singular kernels, and to study their influence on well-posedness, decay rates, and long-time dynamical behavior.
\end{itemize}
\bibliographystyle{plain}
\bibliography{Timoshenko_references}
\end{document}